\documentclass[12pt,reqno]{amsart}
\pdfoutput=1
\usepackage[headings]{fullpage}
\usepackage{amssymb,amsmath,bbm,tikz}
\usepackage{graphicx,color,enumerate, float}
\usepackage[all,cmtip]{xy}
\usepackage{url}
\usepackage[margin=1.2in]{geometry}
\usepackage{verbatim}
\usepackage{pb-diagram}

\usepackage[bookmarks=true,%
    colorlinks=true,%
    linkcolor=blue,%
    citecolor=blue,%
    filecolor=blue,%
    menucolor=blue,%
    urlcolor=blue,%
    breaklinks=true]{hyperref}

\newtheorem{theorem}{Theorem}[section]
\theoremstyle{definition}
\newtheorem{proposition}[theorem]{Proposition}
\newtheorem{lemma}[theorem]{Lemma}

\newtheorem{remark}[theorem]{Remark}

\newtheorem{question}[theorem]{Question}

\renewcommand{\setminus}{{\smallsetminus}}

\newcommand{\Z}{{\mathbb{Z}}}

\newcommand{\C}{{\mathbb{C}}}
\newcommand{\QQ}{{\mathbb{Q}}}

\newcommand{\psl}{{\mathrm{PSL}_2(\mathbb{C})}}

\def\BN{\mathbbm N}
\def\BH{\mathbbm H}
\def\BZ{\mathbbm Z}
\def\BQ{\mathbbm Q}

\def\BC{\mathbbm C}

\def\calT{\mathcal T}

\def\calG{\mathcal G}

\def\ti{\widetilde}
\def\SL{\mathrm{SL}}

\def\PSL{\mathrm{PSL}}

\def\be{\begin{equation}}
\def\ee{\end{equation}}
\def\Ann{\mathrm{Ann}}
\def\ev{\mathrm{ev}}
\def\loc{\mathrm{loc}}
\def\rat{\mathrm{rat}}
\def\llbracket{[\![}
\def\rrbracket{]\!]}
\def\vc{{\vec c}}

\graphicspath{{figures/}}

\begin{document}

\title[A diagrammatic approach to the AJ Conjecture]{
A diagrammatic approach to the AJ Conjecture} 

\author{Renaud Detcherry}
\address{Max Planck Institute for Mathematics \\
         Vivatsgasse 7, 53111 Bonn, Germany \newline
         {\tt \url{http://people.mpim-bonn.mpg.de/detcherry}}
         }
\email{detcherry@mpim-bonn.mpg.de}
\author{Stavros Garoufalidis}
\address{Max Planck Institute for Mathematics \\
         Vivatsgasse 7, 53111 Bonn, Germany \newline
         {\tt \url{http://www.math.gatech.edu/~stavros}}
         }
\email{stavros@math.gatech.edu}
\thanks{
1991 {\em Mathematics Classification.} Primary 57N10. Secondary 57M25.
}

\date{Friday 1 March, 2019}

\begin{abstract} 
The AJ Conjecture relates a quantum invariant, a minimal order recursion
for the colored Jones polynomial of a knot (known as the $\hat{A}$ 
polynomial), with a classical invariant, namely the defining polynomial $A$
of the $\psl$ character variety of a knot. More precisely, the AJ Conjecture
asserts that the set of irreducible factors of the $\hat{A}$-polynomial 
(after we set $q=1$, and excluding those of $L$-degree zero) coincides with 
those of the 
$A$-polynomial. In this paper, we introduce a version of the 
$\hat{A}$-polynomial that depends on a planar diagram of a knot (that
conjecturally agrees with the $\hat{A}$-polynomial) and we prove that it 
satisfies one direction of the AJ Conjecture. Our proof uses
the octahedral decomposition of a knot complement obtained from a planar 
projection of a knot, the $R$-matrix state sum formula for the colored 
Jones polynomial, and its certificate.
\end{abstract}

\maketitle

{\footnotesize
\tableofcontents
}


\section{Introduction}
\label{sec.intro}

\subsection{The colored Jones polynomial and the AJ Conjecture}
\label{sub.jw}

The Jones polynomial of a knot~\cite{Jones} is a powerful knot invariant
with deep connections with quantum field theory, discovered by 
Witten~\cite{Witten:CS}. The discoveries of Jones and Witten gave rise to 
Quantum Topology. An even more powerful invariant is the colored Jones 
polynomial $J_K(n) \in \BZ[q^{\pm 1}]$ of a knot $K$, a sequence of Laurent
polynomials that encodes the Jones polynomial of a knot and its parallels. 
Since the dependence of the colored Jones polynomial $J_K(n)$ on 
the variable $q$ plays no role in our paper, we omit it from the notation.  
The colored Jones polynomial 
determines the Alexander polynomial~\cite{B-NG}, is conjectured to determine
the volume of a hyperbolic knot~\cite{K95, K97, MM}, is conjectured to 
select two out of finitely many slopes of incompressible surfaces of the knot
complement~\cite{Ga:slope}, and is expected to determine the $\SL(2,\BC)$ 
character variety of the knot, viewed from the boundary~\cite{Ga:AJ}.
The latter is the AJ Conjecture, which is the focus of our paper. 

The starting point of the AJ Conjecture~\cite{Ga:AJ} is the fact that 
the colored Jones polynomial $J_K(n)$ of a knot $K$ is 
$q$-holonomic~\cite{gale2}, that is, it satisfies a nontrivial 
linear recursion relation
\be
\label{eq.Jrec}
\sum_{j=0}^d c_j(q,q^n) J_K(n+j) =0, \qquad \text{for all} \,\,\, n \in \BN \,,
\ee
where $c_j(u,v) \in \BZ[u,v]$ for all $j$. 
We can write the above equation in operator form as follows $PJ_K=0$
where $P=\sum_j c_j(q,Q)E^j$ is an element of the ring 
$\BZ[q,Q]\langle E \rangle$ where $EQ=qQE$ are the operators that act on
sequences of functions $f(n)$ by:
\be
\label{eq.EQ}
(Ef)(n)=f(n+1), \qquad (Qf)(n)=q^n f(n) \,.
\ee
Observe that the set 
\be
\label{eq.annf}
\Ann(f)=\{ P \in \BZ[q,Q]\langle E \rangle \, | \, Pf=0 \}
\ee
is a left ideal of $\BZ[q,Q]\langle E \rangle$, nonzero when $f$ is
$q$-holonomic. Although the latter ring is not a principal left ideal domain,
its localization $\BQ(q,Q)\langle E \rangle$ is, and cleaning denominators
allows one to define a minimal $E$-order, content-free element 
$\hat A_K(q,Q,E) \in \BZ[q,Q]\langle E \rangle$ which annihilates the
colored Jones polynomial.

On the other hand, the $A$-polynomial of a knot~\cite{CCGLS} 
$A_K(L,M) \in \BZ[L,M]$ is the defining polynomial for the 
character variety of $\SL(2,\BC)$ representations of the boundary of the 
knot complement that extend to representations of the knot complement. 

The AJ Conjecture asserts that the irreducible factors of $\hat A_K(1,Q,E)$
of positive $E$-degree coincide with those of $A_K(Q,E^{-2})$.
The AJ Conjecture is known for most 2-bridge knots,
and some 3-strand pretzel knots; see~\cite{Le:2bridge} and \cite{LZ:AJ}. 

Let us briefly now discuss the $q$-holonomicity of the colored Jones 
polynomial~\cite{gale2}: this follows naturally 
from the fact that the latter can be expressed as a state-sum formula 
using a labeled, 
oriented diagram $D$ of the knot, placing an $R$-matrix at each crossing 
and contracting indices as described for instance in Turaev's 
book~\cite{Tu:book}. For a diagram $D$ with $c(D)$ crossings, this leads to
a formula of the form
\be
\label{eq.Jw0}
J_K(n)=\sum_{\BZ^{c(D)+1}} w_D(n,k)
\ee
where the summand $w_D(n,k)$ is a $q$-proper hypergeometric function and
for fixed $n$, the support of the summand is a finite set. The fundamental
theoreom of $q$-holonomic functions of Wilf-Zeilberger~\cite{WZ} 
concludes that $J_K(n)$ is $q$-holonomic. 
Usually this ends the benefits of~\eqref{eq.Jw0}, aside from its sometimes 
use as a means of computing some values of the colored Jones
polynomial for knots with small (eg $12$ or less) number of crossings
and small color (eg, $n < 10$). 

Aside from quantum topology, and key to the results of our paper, is
the fact that a planar projection $D$ of a knot $K$ gives rise to an
ideal octahedral decomposition of its complement minus two spheres, and 
thus to a gluing equations variety $\calG_D$ and to an $A$-polynomial $A_D$
reviewed in Section~\ref{sec.ge} below. In~\cite{Kim:octI}, 
Kim-Kim-Yoon prove that $A_D$ coincides with the $A$-polynomial of $K$, and 
in~\cite{Kim:octIII} Kim-Park prove that $\calG_D$ is, up to birational 
equivalence, invariant under Reidemeister moves, and forms a diagrammatic
model for the decorated $\PSL(2,\BC)$ character variety of the knot.

The aim of the paper is to highlight the fact that formulas of the 
form~\eqref{eq.Jw0} lead to further knot invariants which are natural 
from the point of view of holonomic modules and form a rephrasing of the
AJ Conjecture that connects well with the results of~\cite{Kim:octI} 
and~\cite{Kim:octIII}.


\subsection{$q$-holonomic sums}
\label{sub.qholosums}

To motivate our results, consider a sum of the form
\be
\label{eq.fF}
f(n) = \sum_{k \in \BZ^r} F(n,k)
\ee
where $n \in \BZ$ and $k=(k_1,\dots,k_r) \in \BZ^r$ and $F(n,k)$ is a 
proper $q$-hypergeometric function with compact support for fixed $n$. 
Then $f$ is $q$-holonomic but more is true. The annihilator 
$$
\Ann(F) \subset \BQ[q,Q,Q_k]\langle E, E_k \rangle
$$
of the summand 
is a $q$-holonomic left ideal where $E_k=(E_{k_1},\dots,E_{k_r})$ and 
$Q_k=(Q_{k_1},\dots,Q_{k_r})$ are operators, each acting in one of the $r+1$
variables $(n,k)$ with the obvious commutation relations (operators
acting on different variables commute and the ones acting on the same variable 
$q$-commute). Consider the map
\be
\label{eq.phi}
\varphi: \BQ[q,Q]\langle E, E_k \rangle \to \BQ[q,Q]\langle E \rangle,
\qquad \varphi(E_{k_i})=1, \,\, i=1,\dots, r \,.
\ee
It is a fact (see Proposition~\ref{prop:cert} below) that 
\be
\label{eq.cert}
\varphi( \Ann(F) \cap \BQ[q,Q]\langle E, E_k \rangle ) \subset \Ann(f) 
\ee
and that the left hand side is nonzero. Elements of the left hand side
are usually called ``good certificates'', and in practice one uses the above
inclusion to compute a recursion for the sum~\cite{AB,Zeil:creative}.
If $\hat A^c_F(q,Q,E)$ and $\hat A_f(q,Q,E)$ denotes generators of the 
left and the right hand side of~\eqref{eq.cert}, it follows that 
$\hat A_f(q,Q,E)$ is a right divisor of $\hat A^c_F(q,Q,E)$. We will call
the latter the certificate recursion of $f$ obtained from~\eqref{eq.fF}.

In a sense, the certificate recursion of $f$ is more natural than the
minimal order recursion and that is the case for holonomic $D$-modules 
and their push-forward, discussed for instance by Lairez~\cite{Lairez}.

What is more important for us is that 
if one allows presentations of $f$ of the form~\eqref{eq.fF} where $F$
is allowed to change by for instance, consequences of the $q$-binomial
identity, then one can obtain an operator $\hat A^c_f(q,Q,E)$ which 
is independent of the chosen presentation. 

\subsection{Our results}
\label{sub.results}

Applying the above discussion to~\eqref{eq.Jw0} with $F=w_D$, allows 
us to introduce the certificate recursion 
$\hat A_D^c(q,Q,E) \in \BZ[q,Q]\langle E \rangle$
of the colored Jones polynomial, which depends on a labeled, oriented 
planar diagram $D$ of a knot.
We can also define $\hat A_K^c(q,Q,E) \in \BZ[q,Q]\langle E \rangle$ to be the
left gcd of the elements $\hat A_D^c$ in the local ring 
$\BQ(q,Q)\langle E \rangle$, lifted back to $\BZ[q,Q]\langle E \rangle$.

We now have all the ingredients to formulate one direction of a refined AJ 
Conjecture. Our proof uses
the octahedral decomposition of a knot complement obtained from a planar 
projection of a knot, the $R$-matrix state sum formula for the colored 
Jones polynomial, and its certificate.

\begin{theorem}
\label{thm.1}
For every knot $K$, 
\newline
{\rm (a)} $\hat A_K$ divides  $\hat A_K^c$.
\newline
{\rm (b)} Every irreducible factor of $A_K(Q,E^{-2})$ 
of positive $E$-degree is a factor of $\hat A_K^c(1,Q,E)$.
\end{theorem}

\begin{remark}
The $\hat A_K$-polynomial has only been computed in a handful of cases, 
see \cite{GS:twist}, \cite{GM:pretzel}, \cite{GK:pretzel} and \cite{GK:74}. 
In all cases where $\hat A_K$ is known, it is actually 
obtained from certificates and in that case $\hat A_K^c=\hat A_K$.
\end{remark}

\begin{question}
Is it true that for any knot $K$, 
one has $\hat A_K^c=\hat A_K$ ?
\end{question}

\begin{question} 
Is it true that the certificate recursion $\hat A^c_D$ of a planar
projection of a knot is invariant under Reidemester moves on $D$?
\end{question}

A positive answer to the latter question is a quantum analogue of the fact 
that the gluing equation variety $\calG_D$ associated to a diagram $D$ is 
independent of $D$, a result that was announced by Kim and 
Park~\cite{Kim:octIII}. We believe that the above question has a positive
answer, coming from the fact that the Yang-Baxter equation for the R-matrix
follows from a $q$-binomial identity, 
but we will postpone this investigation to a future publication.

\subsection{Sketch of the proof}
\label{sub.sketch}

To prove Theorem \ref{thm.1}, we fix a planar projection $D$ of an oriented 
knot $K$. On the one hand, the planar projection gives rise to an
ideal decomposition of the knot complement (minus two points) using one
ideal octahedron per crossing, subdividing further each octahedron 
to five ideal tetrahedra. This ideal decomposition gives rise to
a gluing equations variety, discussed in Section~\ref{sec.ge}. 
On the other hand, the planar projection gives a state-sum for the 
colored Jones polynomial, by placing one $R$-matrix per crossing and 
contracting indices. The summand of this state-sum is $q$-proper 
hypergeometric and its annihilator defines an ideal in a quantum
Weyl algebra, discussed in Section~\ref{sec.CJ}. The annihilator ideal
is matched when $q=1$ with the gluing equations ideal in the key
Proposition~\ref{thm.match}. This matching, implicit in the Grenoble notes
of D. Thurston~\cite{Thu:grenoble}, combined with a certificate (which is
a quantum version of the projection map from gluing equations variety to
$\BC^* \times \BC^*$), and with the fact that the gluing equation variety sees
all components of the $\PSL(2,\BC)$ character variety~\cite{Kim:octI},
conclude the proof of our main theorem. 

Our method of proof for Theorem~\ref{thm.1} using certificates to show one
direction of the AJ Conjecture is general and flexible and can be applied 
in numerous other situations, in particular to a proof of one direction
of the AJ Conjecture for state-integrals, and to one direction 
of the AJ Conjecture for the 3Dindex~\cite{And,Dimofte}. This will be
studied in detail in a later publication. For a discussion
of the AJ Conjecture for state-integrals and for a proof in the case of the
simplest hyperbolic knot, see~\cite{And}.

Finally, our proof of Theorem~\ref{thm.1} does not imply any relation
between the Newton polygon of the $\hat A_K(q,Q,E)$ polynomial and that
of $A_K(1,Q,E)$. If the two Newton polygons coincided, the Slope Conjecture
of~\cite{Ga:slope} would follow, as was explained in~\cite{Ga:quadratic}.
Nonetheless, the Slope Conjecture is an open problem.
 

\section{Knot diagrams, their octahedral decomposition and 
their gluing equations}
\label{sec.ge}

\subsection{Ideal triangulations and their gluing equations}
\label{sec:gluingeq-spines}

Given an ideal triangulation $\calT$ of a 3-manifold $M$ with cusps, 
Thurston's gluing equations (one for each edge of $\calT$) 
give a way to describe the hyperbolic structure on $M$ and its
deformation if $M$ is hyperbolic~\cite{Th,NZ}. The gluing equations define  
an affine variety $\calG_\calT$, the so-called gluing equations variety, 
whose definition we now recall. The edges of each combinatorial ideal 
tetrahedron get assigned variables, with opposite edges having the same
variable as in the left hand side of Figure~\ref{fig:dualspine}. 
The triple of variables (often called a triple of shapes of the tetrahedron) 
$$
(z,z',z'')=\left(z, \frac{1}{1-z}, 1-\frac{1}{z} \right) 
$$
satisfies the equations
\be
\label{eq.triplez}
z z' z'' = -1, \qquad z z' -z + 1  = 0 \,
\ee
and every solution of~\eqref{eq.triplez} uniquely defines a triple of
shapes of a tetrahedron. Note that the shapes of the tetrahedron 
$z$, $z'$, or $z''$ lie in $\BC^{**}=\BC\setminus\{0,1\}$, and that they
are uniquely determined by $z \in \BC^{**}$. When we talk about assigning
a shape $z$ to a tetrahedron below, it determines shapes $z'$ and $z''$
as in Figure~\ref{fig:dualspine}.

\begin{figure}[!htpb]
\centering
\def \svgwidth{.8\columnwidth}
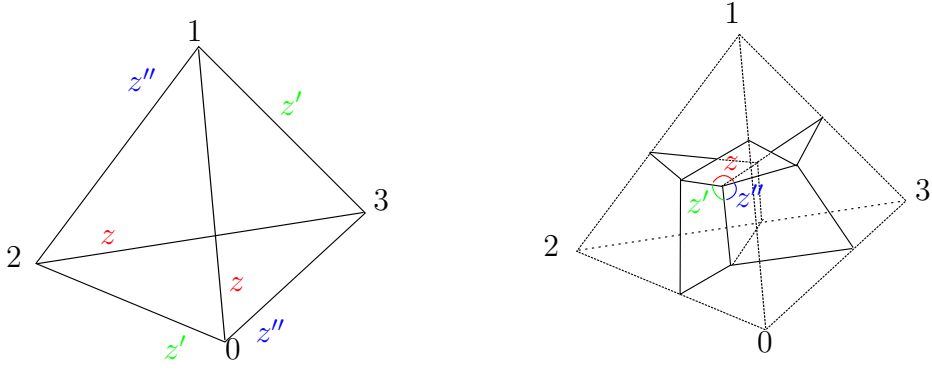
\caption{The dual spine to the triangulation and the shape parameters 
associated to corners of the spine.}
\label{fig:dualspine}
\end{figure}

Given an ideal triangulation $\calT$ with $N$ tetrahedra, assign shapes
$z_i$ for $i=1,\dots,N$ to each tetrahedron. If $e$ is an edge of $\calT$
the corresponding gluing equation is given by
$$
\underset{\Delta \in N(e)}{\prod} z_{\Delta}=1 \,,
$$
where $N(e)$ is the set of all tetrahedra that meet along the edge $e$, 
and $z_{\Delta}$ is the shape parameter corresponding to the edge $e$ of 
$\Delta$. The gluing equation variety $\calG_\calT$ is the affine variety 
in the variables $(z_1,\dots,z_N) \in (\BC^{**})^N$ defined by the edge 
gluing equations, for all edges of $\calT$. Equivalently, it is the 
affine variety in the variables $(z_1,z_1',z_1'',\dots,z_N,z_N',z_N'') 
\in \BC^{3N}$ defined by the edge equations and the 
equations~\eqref{eq.triplez}, one for each tetrahedron.

We next discuss the relation between a solution to the gluing equations
and decorated (or sometimes called, augmented) $\PSL(2,\BC)$ representations 
of the fundamental group of the underlying 3-manifold $M$. The construction
of decorated representations from solutions to the gluing equations
appears for instance in Zickert's thesis~\cite{Zickert:thesis} and also 
in~\cite{GGZ}. Below, we follow the 
detailed exposition by Dunfield given in~\cite[Sec.10.2-10.3]{BDRV}. 

A solution of the gluing equations gives rise to a developing map
$ \ti M \to \BH^3$ from the universal cover $\ti M$ to the 3-dimensional
hyperbolic space $\BH^3$. Since the orientation preserving isometries of
$\BH^3$ are in $\PSL(2,\BC)$, this in turn gives rise to a 
$\PSL(2,\BC)$ representation of the fundamental
group $\pi_1(M)$, well-defined up to conjugation. What's more, we get 
a decorated representation (those were called augmented representations in 
Dunfield's terminology). Following the notation 
of~\cite[Sec.10.2-10.3]{BDRV}, let $\overline{X}(M,\PSL(2,\BC))$ 
denote the augmented character variety of $M$. Thus, we get a map:
\be
\label{eq.dev}
\calG_\calT \to \overline{X}(M,\PSL(2,\BC)) \,.
\ee
So far, $M$ can have boundary components of arbitrary genus. 
When the boundary $\partial M$ consists of a single torus boundary component, 
and $\gamma$ is a simple closed curve on $\partial M$, 
the holonomy of an augmented representation gives a regular function 
$h_\gamma:  \overline{X}(M,\PSL(2,\BC)) \to \BC^*$.  
Note that for a decorated representation $\rho$, the set
of squares of the eigenvalues of $\rho(\gamma) \in \PSL(2,\BC)$ is given
by $\{h_\gamma(\rho), h_\gamma(\rho)^{-1}\}$. Once 
we fix a pair of meridian and longitude $(\mu,\lambda)$ of the boundary 
torus, then we get a map 
\be
\label{eq.lm}
(h_\mu, h_\lambda) : \overline{X}(M,\PSL(2,\BC)) \to \BC^* \times \BC^* \,.
\ee
The defining polynomial of the 1-dimensional components of the above
map is the $A$-polynomial of the 3-manifold $M$. Technically, this is
the $\PSL(2,\BC)$-version of the $A$-polynomial and its precise relation with
the $\SL(2,\BC)$-version of the $A$-polynomial (as defined by~\cite{CCGLS})
is discussed in detail in Champanerkar's thesis~\cite{Ch}; 
see also~\cite[Sec.10.2-10.3]{BDRV}.

We should point out that although~\eqref{eq.dev} is a map of affine varieties, 
its image may miss components of $\overline{X}(M,\PSL(2,\BC))$, and
hence the gluing equations of the triangulation may not detect some factors
of the $A$-polynomial. In fact, when the boundary of $M$ consists of 
tori, the image of~\eqref{eq.dev} always misses 
the components of abelian $\SL(2,\BC)$ representations (and every knot 
complement has a canonical such component), but it may also
miss others. For instance, there is a 5-tetrahedron ideal triangulation of the
$4_1$ knot with an edge of valency one, and for that triangulation, 
$\calG_\calT$ is empty. 

For later use, let us record how to compute the holonomy of a peripheral
curve on the gluing equations variety. Given a path $\gamma$ in a 
component of $\partial M$ that is normal with respect to this triangulation, 
it intersects the triangles of $\partial M$ in segment joining different 
sides. Each segment may go from one side of the triangle to either the 
adjacent left side or right side. Also it separates one corner of the 
triangle from the other two; this corner correspond to a shape parameter 
which we name $z_{left}$ or $z_{right}$ depending whether the segment goes 
left or right. The holonomy of $\gamma$ is then:
$$
h_\gamma =\underset{left \  segments}{\prod} z_{left} 
\underset{right \  segments}{\prod} z_{right}^{-1}.
$$ 

\subsection{Spines and gluing equations}
\label{sub.spine}

The ideal triangulations that we that we will discuss in the next section
come from a planar projection of a knot, and it will be easier to
work with their spines, that is the the dual $2$-skeleton. Because of this
reason, we discuss the gluing equations of an ideal triangulation $\calT$ 
in terms of its spine.
In that case, edges of $\calT$ are dual to 2-cells of the spine, and give
rise to gluing equations. Recall that a spine $S$ of $M$ is a 
CW-complex embedded in $M$, such that each point of $S$ has a neighborhood 
homeomorphic to either $D^2$, $Y\times [0,1]$ where $Y$ is the $Y$-shaped 
graph or to the cone over the edges of a tetrahedron, and such that 
$M\setminus S$ 
is homeomorphic to $\partial M \times [0,1)$. Points of the third type 
are vertices of the spine, points of the second type form the edges of 
the spines and points of the first type form the regions of the spine.

\begin{figure}[!htpb]
\centering
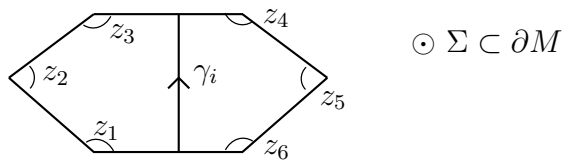
\caption{A segment $\gamma_i$ of a peripheral loop $\gamma$ intersecting 
a region of the spine. The boundary component $\Sigma$ to which $\gamma$ 
belongs lies above the region. In this example, 
$h_{\gamma_i}=-z_1z_2z_3=-\frac{1}{z_4z_5z_6}$.}
\label{fig:holonomy}
\end{figure}

For any ideal triangulation of $M$, the dual spine is obtained as shown in 
Figure~\ref{fig:dualspine}. Shape parameters that were assigned to 
tetrahedra are now assigned to vertices of the spine. At each vertex, 
two opposite corners bear the same shape parameter $z$, and the other 
bear the parameters $z',z''$ according to the cyclic ordering (see 
Figure~\ref{fig:dualspine}). Edge equations translate into region equations, 
the region equation associated to the region $R$ being:
$$
\underset{c \in \textrm
{corners}(R)}{\prod} z_c=1.
$$
For a path $\gamma$ on the spine $S$ that is in normal position with 
respects to $S$, it intersects each region in a collection of segments 
$(\gamma_i)_{i \in I}$. The holonomy of the segment $\gamma_i$ is 
$$
h_{\gamma_i}=-\underset{c \ \textrm{left corner}}{\prod} z_c 
=-\underset{c \ \textrm{right corner}}{\prod} z_c^{-1},
$$
where left and right corners are defined as in Figure~\ref{fig:holonomy}, 
and the holonomy of $\gamma$ is
$$
h_{\gamma}=\underset{i\in I}{\prod}h_{\gamma_i}.
$$

\subsection{The octahedral decomposition of a knot diagram}
\label{sec:5t-gluingequations}

\begin{figure}[!htpb]
\centering
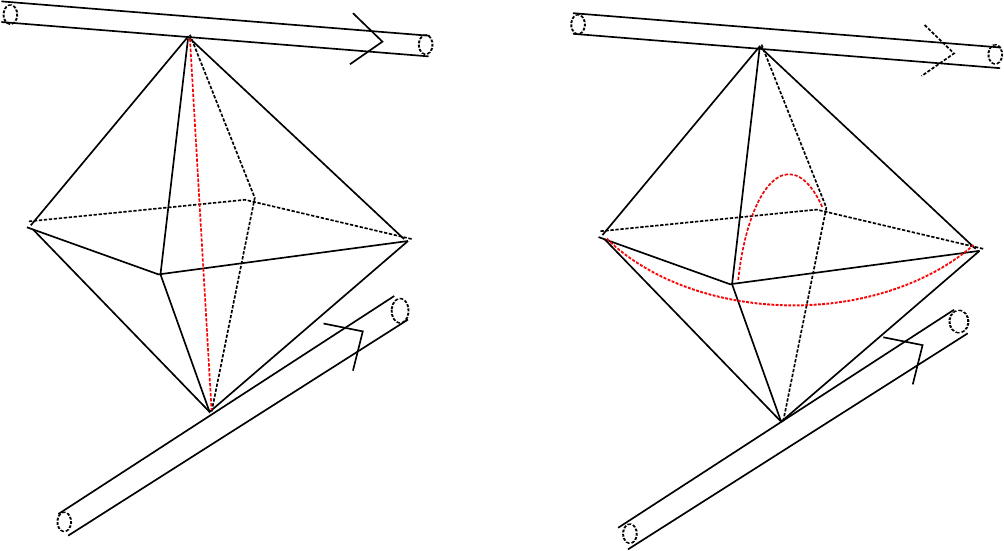
\caption{Any octahedron can be split into $4$ or $5$ tetrahedra by adding 
the red dashed edges to it.}
\label{fig:octahedron}
\end{figure}

In this section we fix a diagram $D$ in $S^2$ of an oriented knot $K$. 
By diagram, we mean an embedded 4-valent graph in the plane, with an 
overcrossing/undercrossing choice at each vertex.
Let $X(D)$ and $c(D)$ denote the set and the number of crossings of $D$. 
In this section 
as well as the remainder of the paper, an arc of $D$ will be the segment 
of the diagram joining two successive crossings of $D$. 
An \textit{overpass} (resp. \textit{underpass}) will be a small portion of 
the upper strand (resp. lower strand) of a crossing. We will denote the 
set of overpasses by $O(D)$ and the set of underpasses by $U(D)$. 
An \textit{overarc} (resp. \textit{underarc}) will be the portion of the 
knot joining two successive underpasses (resp. overpasses). An overarc of 
$K$ may pass through some number of crossings of $K$, doing so as the upper 
strand each time.

Given a diagram $D$ of the knot $K$ with $c(D)$ crossings, let $B_1$ be 
some ball lying above the projection plane and $B_2$ another ball lying 
under the projection plane. A classical construction, first introduced 
by Weeks in his thesis, and implemented in \texttt{SnapPy}
as a method of constructing ideal triangulations of planar projections of 
knots~\cite{snappy, Weeks}, 
yields a decomposition of $S^3\setminus(K\cup B_1 \cup B_2)$ 
into $c(D)$ ideal octahedra. The decomposition works as follows: at each 
crossing of $K$, put an octahedron whose top vertex is on the overpass and 
bottom vertex is on the underpass. Pull the two middle vertices lying on 
the two sides of the overpass up towards $B_1$ and the two other middle 
vertices down towards $B_2$. One can then patch all these octahedra 
together to get a decomposition of $S^3\setminus (K\cup B_1 \cup B_2)$. 
We refer to \cite{Kim:octI} as well as \cite{Thu:grenoble} 
for figures and more details on this construction.

From the octahedral decomposition of $S^3\setminus (K\cup B_1 \cup B_2)$, 
one can get an ideal triangulation of $S^3\setminus (K\cup B_1 \cup B_2)$ 
simply by splitting the octahedra further into tetrahedra. 
There are two natural possibilities for this splitting, as one can cut each 
octahedra into either $4$ or $5$ tetrahedra as shown in 
Figure~\ref{fig:octahedron}. We will be interested in the decomposition 
where we split each octahedra into $5$ tetrahedra, obtaining thus a 
decomposition of $S^3\setminus (K\cup B_1 \cup B_2)$ into $5c(D)$ tetrahedra. 
We denote this ideal triangulation by $\calT^{5T}_D$, and we call it the
``$5T$-triangulation of $D$''.

Since the inclusion map $S^3\setminus (K \cup B_1 \cup B_2) \to S^3\setminus K$
is an isomorphism on fundamental groups, a solution to the gluing equations
of $\calT^{5T}_D$ gives rise to a decorated $\PSL(2,\BC)$ representation
of the knot complement.

\subsection{The spine of the $5T$-triangulation of a knot diagram
and its gluing equations}

Let $\calG_D$ denote the gluing equation variety of $\calT^{5T}_D$.
To write down the equations of $\calG_D$, we will work with the dual 
spine, and use the spine formulation of the gluing 
equations introduced in Section~\ref{sec:gluingeq-spines}. We describe this 
spine just below. 
This well-known spine is studied in detail by several authors 
including~\cite{Kim:octI}.

\begin{figure}[!htpb]
\centering
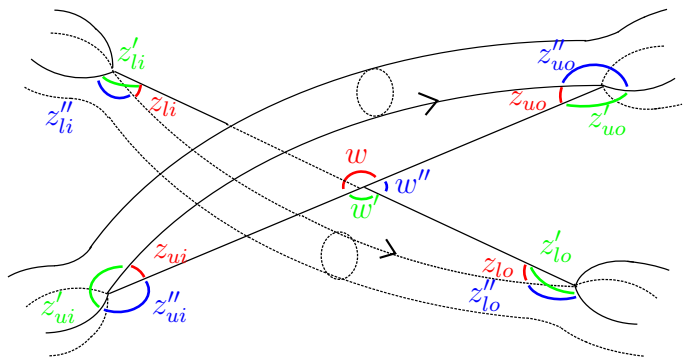
\caption{The $5T$-spine near a crossing of $D$, and the shape parameters 
of each corner of the spine. The arrows specify the orientation of 
strands.}
\label{fig:5tspine}
\end{figure}

Figure~\ref{fig:5tspine} shows a picture of the spine near a crossing 
of $D$. The spine contains $5$ vertices near each crossing of $D$ and can 
be described as follows:

First we embed $K$ in $S^3$ as a solid torus sitting in the middle of 
the projection plane; except for overpasses which go above the projection 
plane and underpasses which go below. We let the boundary of a tubular
neighborhood of $K$ to be a subset of 
the spine. At each crossing we connect the overpass and the underpass 
using two triangles that intersects transversally in one point. Finally 
we glue the regions of the projection plane that lie outside $D$ to the 
rest of the spine. The regions of the spine are then of $3$ types: 

\begin{itemize}
\item
An upper/lower triangle region for each crossing, and $2c(D)$ in total.
\item
For each region of $D$ one gets an horizontal 
region in the spine; we call these \textit{big regions}, $c(D)+2$ in total.
\item
The boundary of a neighborhood of $K$ 
is cut by the triangle regions and the big regions into 
regions lying over the projection plane (\textit{upper shingle region}) and 
some lying under the projection plane (\textit{lower shingle regions}). 
Note that upper shingle regions start and end at underpasses; they are in 
correspondance with the overarcs of the diagram, $c(D)$ in total. 
Similarly, the lower shingle regions are in correspondance with underarcs, 
and there is also $c(D)$ of them.
\end{itemize}

We now assign shape parameters to each vertex of the spine as shown 
in Figure~\ref{fig:5tspine}. There are $5$ shape parameters for each 
crossing $c$: a central one which we call $w_c$ and $4$ others: 
$z_{c,li},z_{c,lo},z_{c,ui},z_{c,uo}$ standing for lower-in, lower-out, 
upper-in and upper-out. When the crossing $c$ we consider is clear, 
we will sometimes write $w,z_{li},z_{lo}\ldots$ dropping the index $c$.

Note that the assignment of shape parameters is such that the main 
version of the parameter $w,z_{li},\ldots$ lies on a corner of a triangle 
region, while the auxiliary $w',w'',z_{li}',z_{li}''\ldots$ are prescribed 
by the cyclic ordering induced by the boundary of 
$S^3\setminus (K\cup B_1 \cup B_2)$.

We can now write down the gluing equations coming from the $5T$-spine:

$\bullet$ The upper/lower triangle equations are (in the notation of
Figure~\ref{fig:5tspine})
\begin{equation}
\label{eq:uppertriangle}
w z_{ui} z_{uo}=1, \qquad w z_{li} z_{lo}=1.
\end{equation}

\begin{figure}[!htpb]
\centering
\def \svgwidth{.9\columnwidth}
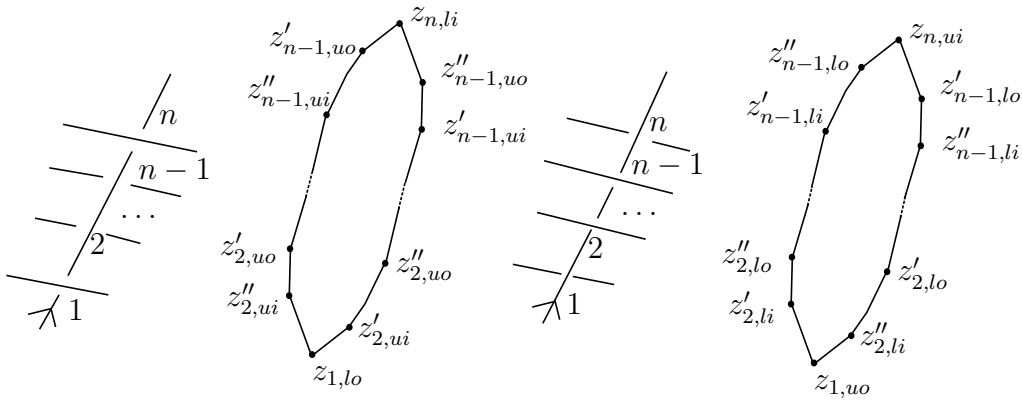
\caption{An overarc (resp. underarc) and the corresponding upper 
(resp. lower) shingle region of the spine, with shape parameters.}
\label{fig:shingleregion}
\end{figure}

$\bullet$ The upper/lower shingle equations. Consider an upper shingle region 
corresponding to an overarc going from some crossing labelled $1$ to the 
crossing $n$, going through crossings $1,2,\ldots ,n-1$ as overpasses. Then 
the shingle region has one corner for each of its ends, and $4$ corners for 
each overpasses, as explained in Figure~\ref{fig:shingleregion}. We get:
$$
z_{1,lo}z_{2,ui}'z_{2,uo}'' \ldots 
z_{n-1,ui}' z_{n-1,uo}'' z_{n,li} z_{n-1,uo}'z_{n-1,ui}'' 
\ldots z_{2,uo}'z_{2,ui}''=1 \,.
$$ 

\begin{lemma}
\label{lem.eshingle}
The upper/lower shingle equations have the equivalent forms, respectively:
\be
\label{eq:uppershingle}
z_{n,lo} = z_{1,lo}w_n^{-1}\underset{j=2}{\overset{n-1}{\prod}} w_j\,, 
\qquad
z_{n,li} = z_{1,li} w_1 \underset{j=2}{\overset{n-1}{\prod}} w_j^{-1} \,.
\ee
\be
\label{eq:lowershingle}
z_{n,ui} = z_{1,ui} w_1 \underset{j=2}{\overset{n-1}{\prod}} w_j^{-1}\,,
\qquad 
z_{n,uo} = z_{1,uo} w_n^{-1} \underset{j=2}{\overset{n-1}{\prod}} w_j \,.
\ee
\end{lemma}

\begin{proof}
Grouping together shape parameters coming from the same vertex and using 
$zz'z''=-1$, we get:
$$
z_{n,li}z_{1,lo}=\underset{j=2}{\overset{n-1}{\prod}} z_{j,ui}z_{j,uo}
$$
and then, using Equation~\eqref{eq:uppertriangle}:
$$
z_{n,li}z_{1,lo}= \underset{j=2}{\overset{n-1}{\prod}} w_j^{-1}
$$
Finally, using Equation~\eqref{eq:uppertriangle}, we can rewrite this as 
equation~\eqref{eq:uppershingle} 
between only $z_{lo}$'s (or only $z_{li}$'s) parameters.

Similarly for a lower shingle region corresponding to an underarc 
running from crossing $1$ to crossing $n$, one gets an equation:
$$
z_{1,uo} z_{2,li}'' z_{2,lo}'\ldots 
z_{n-1,li}''z_{n-1,lo}' z_{n,ui} z_{n-1,lo}'' z_{n-1,li}' 
\ldots z_{2,lo}'' z_{2,li}'=1 \,,
$$
which simplifies to~\eqref{eq:lowershingle}.
\end{proof}

\begin{figure}[!htpb]
\centering
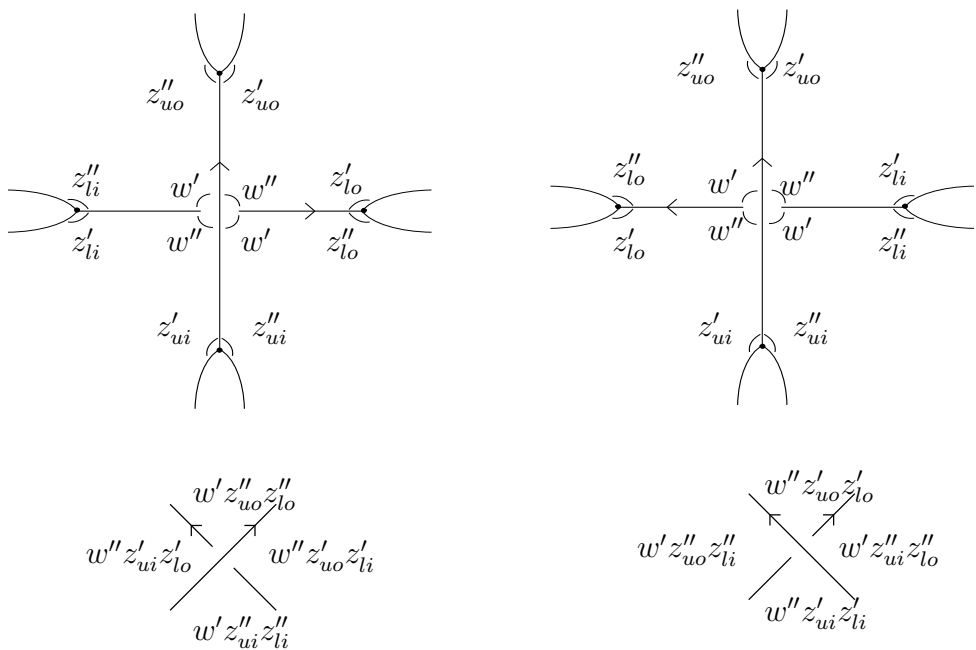
\caption{On the top, a top view of the $5t$-spine near a positive and a 
negative crossing. On the bottom, the rule describing the corner factors.}
\label{fig:cornerfactors}
\end{figure}

$\bullet$ 
Figure~\ref{fig:cornerfactors} shows a top-view of the $5T$-spine 
near a crossing, as well as the shape parameters of horizontal corners of 
the spine. We see that each vertex of a region of $K$ gives rise to $3$ 
corners in the corresponding big region. For each region $R_i$ of $K$, we 
get a big region equation of the form 
\begin{equation}
\label{eq:bigregion}
\underset{v \ \textrm{corner of}\ R_i}{\prod} f(v)=1
\end{equation} 
where the corner factors $f(v)$ are prescribed by the rule shown in 
Figure~\ref{fig:cornerfactors}.

\begin{figure}[!htpb]
\centering
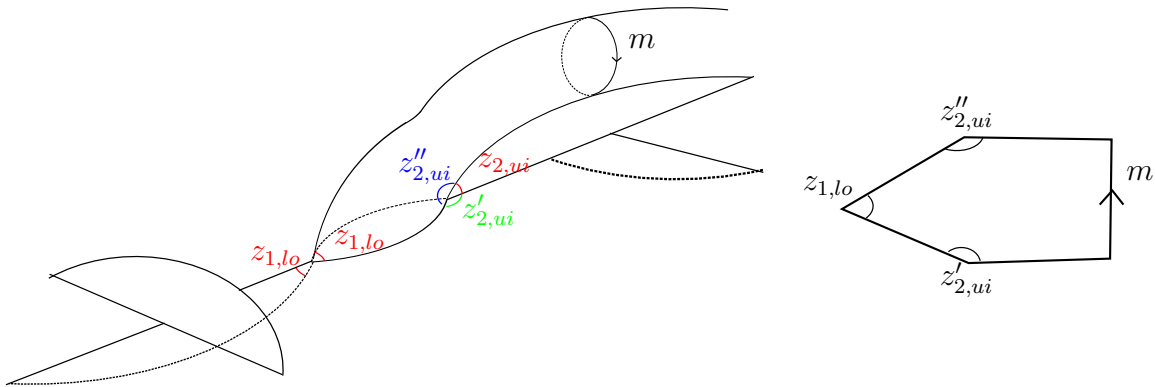
\caption{The meridian positioned on top of overpass $2$, and the left 
part of the region of the $5t$ spine that $m$ intersects.}
\label{fig:meridian}
\end{figure}

\begin{figure}[!htpb]
\centering
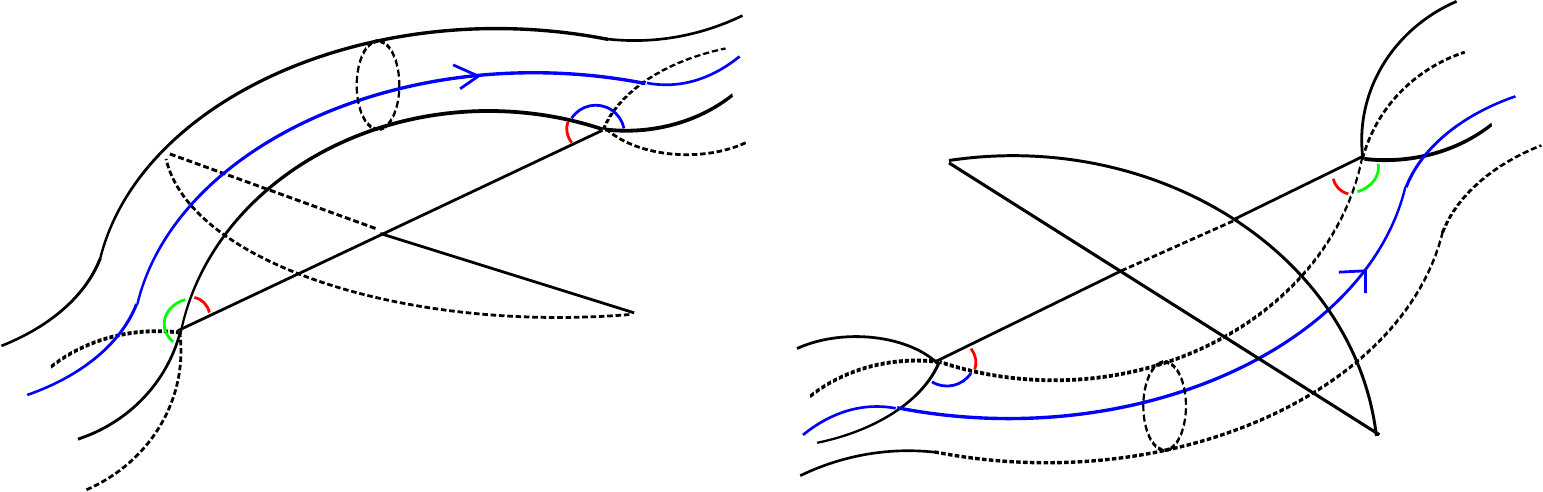
\caption{The longitude $\tilde{l}$ on the $5t$-spine, and the shape 
parameters to the left (resp. to the right) of it on overpasses (resp. 
underpasses).}
\label{fig:longitude}
\end{figure}

Below, we will denote the triangle, region and shingle equations by
$t_i$, $r_k$ and $s_j$ respectively. The above discussion defines 
the gluing equations variety $\calG_D$ as an affine subvariety of
$(\C^{**})^{5c(D)}$ defined by 

\be
\label{eq.calGD}
\calG_D=
\lbrace (w_c,z_{c,ui},z_{c,uo},z_{c,li},z_{c,lo})_{c \in c(D)}
\in (\C^{**})^{5c(D)} \ | \ t_i=1, s_j=1, r_k =1 \rbrace \,.
\ee

We now express the holonomies $w_\mu=h_{\mu}$ and 
$w_\lambda=h_{\lambda}$ of the meridian 
$\mu$ and zero winding number longitude $\lambda$ in terms of the above shape 
parameters. Note that if $K$ is not the unknot, it is always possible to 
find in the diagram of $K$ an underpass that is followed by an overpass 
that corresponds to a different crossing of $K$. We then name those two 
crossings $1$ and $2$. Assume that the meridian is positioned as shown in 
Figure~\ref{fig:meridian}. Then the rule described in 
Section~\ref{sec:gluingeq-spines} gives us the following holonomy:
$$
h_{\mu}=-z_{1,lo}z_{2,ui}'z_{2,ui}''.
$$
As $z_{2,ui}z_{2,ui}'z_{2,ui}''=-1$,  we get:
\begin{equation}
\label{eq:holonomy_meridian}
w_\mu=h_{\mu}=\frac{z_{1,lo}}{z_{2,ui}}.
\end{equation}

Finally, we turn to the holonomy of a longitude. We first compute the 
holonomy of the longitude $\tilde{l}$ corresponding to the blackboard 
framing of the knot. We can represent this longitude on the diagram $D$ 
as a right parallel of $D$. We draw this longitude on the spine in 
Figure~\ref{fig:longitude}, we can see that it intersects each upper or 
lower shingle region in one segment.

We compute the holonomy of each segment in an upper shingle using the 
convention 
$$
h_{a}=-\underset{\textrm{c left corner}}{\prod} z_c
$$
and each lower shingle segment using the convention 
$$
h_{a}=-\underset{\textrm{c right corner}}{\prod} z_c^{-1} \,.
$$
We can actually ignore the $-1$ signs as there are $2c(D)$ segments, an 
even number.

As Figure~\ref{fig:longitude} shows, we get:
$$
h_{\tilde{\lambda}}=\underset{\textrm{overarc} \ a}{\prod}\ 
\underset{\textrm{overpasses} \in a}{\prod} z_{uo}'' z_{ui}' 
\underset{\textrm{underarc} \ a}{\prod}\ \underset{\textrm{underpasses} 
\in a}{\prod} \frac{1}{z_{lo}' z_{li}''}=\underset{X(D)}{\prod} 
\frac{z_{uo}''z_{ui}'}{z_{lo}'z_{li}''}.
$$
The last product is over the set $X(D)$ of crossings of $D$, and for 
simplicity we do not indicate the dependence of the variables on the
crossing $c \in X(D)$. 
Let $\lambda$ be the longitude with zero winding number with $K$. 
The winding number of the blackboard framing longitude $\tilde{\lambda}$ 
is the writhe $\mathrm{wr}(D)$ of the diagram $D$, which can be computed by 
$\mathrm{wr}(D)= c_+ -c_-$, where $c_+$ and $c_-$ are the number of 
positive and negative crossings of the diagram. We then have 
$\tilde{\lambda}=\lambda \mu^{wr(D)}$ and thus
\begin{equation}
\label{eq:hollongitude}
w_\lambda=h_{\lambda}=w_\mu^{-wr(D)}\underset{X(D)}{\prod} 
\frac{z_{uo}''z_{ui}'}{z_{lo}'z_{li}''} \,.
\end{equation}

\subsection{Labeled knot diagrams}
\label{sub.labelD}

In this section we introduce a labeling of the crossings in a knot diagram, 
closely related to the Dowker-Thistlethwaite notation of knots. 

Recall that $D$ is a planar diagram of an oriented knot $K$ and that
we have chosen two special crossings $1$ and $2$ that are successive in 
the diagram, such that such crossing $1$ corresponds to an underpass 
and crossing $2$ to an overpass. This choice determines a labeling of 
crossings of $D$ as follows.

Following the knot, we label the other crossings $3,4,\ldots.$ Note that 
as the knot passes through each crossing twice, each crossing $c$ of $D$ 
gets two 
labels $j<j'$. Exactly one of those two labels correspond to the overpass 
and the other one to the underpass. Arcs of the 
diagram join two successive over- or underpasses labeled $l$ and $l+1$ 
(or $2c(D)$ and $1$). We write $[l,l+1]$ for the arc joining crossings $l$ 
and $l+1$. 

This labeling is illustrated in Figure \ref{fig:crossinglabels} in the case 
of the Figure eight knot.

\begin{figure}[!htpb]
\centering
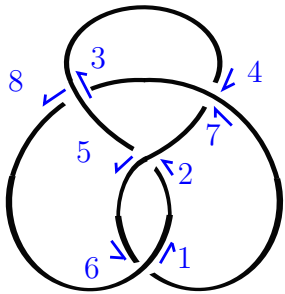
\caption{A labelling of the crossings of a Figure eight knot diagram. 
The $4$ distinct crossings of the diagram have labels $(1,6),(2,5),(3,8)$ 
and $(4,7)$.}
\label{fig:crossinglabels}
\end{figure}

\subsection{Analysis of triangle and shingle relations}
\label{sub.red1}

In this section, we show that the triangle and shingle equations allow us 
to eliminate variables in the gluing variety $\calG_D$. We have the following:

\begin{proposition}
\label{prop.elimination}
In $\calG_D$, each of the variables $w_c,z_{c,li},z_{c,lo},z_{c,ui},z_{c,uo}$ are 
monomials in the variables $w_c,w_{\mu}$ and $w_0=z_{1,lo}$.
\end{proposition}

\begin{proof}
Fix a labeled knot diagram $D$ as in Section~\ref{sub.labelD}. 
Before eliminating variables, we start by 
assigning to each arc $[l,l+1]$ of the diagram a new parameter $z_{l,l+1}$.
These parameters are expressed in terms of the previous parameters by the 
following rules:
$$
z_{1,2}=z_{1,lo}=w_0 \ \textrm{and} \ z_{l,l+1}=z_{1,lo} 
\underset{j \in \llbracket 2,l\rrbracket \cap O(D)}{\prod}w_j 
\underset{j \in \llbracket 2,l\rrbracket \cap U(D)}{\prod}w_j^{-1}.
$$
We recall that in the above $O(D)$ (resp. $U(D)$) is the set of overpasses 
(resp. underpasses) in the diagram $D$. Also, given integers 
$a,b\in \Z$ with $a \leq b$, we denote 
$$
\llbracket a,b \rrbracket = \lbrace a,a+1,\ldots ,b \rbrace \,.
$$
Note that the arc parameters $z_{l,l+1}$ are all clearly monomials in $w_0$ 
and the $w_c$'s.

We claim that each of the shape parameters $z_{c,li},z_{c,lo},z_{c,ui},z_{c,uo}$ 
are monomials in the $z_{l,l+1}$'s and $w_{\mu}$. This will imply the 
proposition. Indeed, let $[k,k+1]$ be an arc of $K$. Then we claim that:
$$
z_{k,k+1}=\begin{cases} z_{k,lo} \ \textrm{if} \ k \ \textrm{is an underpass}
\\ \frac{1}{z_{k+1,li}} \ \textrm{if} \ k+1 \ \textrm{is an underpass}
\\ \frac{w_\mu}{z_{k,uo}} \ \textrm{if} \ k \ \textrm{is an overpass}
\\ \frac{z_{k+1,ui}}{w_\mu} \ \textrm{if} \ k+1 \ \textrm{is an overpass}
\end{cases}
$$

Note $z_{1,2}=z_{1,lo}$ by definition. If $k$ is an underpass, the formula 
$$
z_{k,k+1}=z_{1,lo} \underset{j \in \llbracket 2,k\rrbracket \cap O(D)}{\prod}w_j
\underset{j \in \llbracket 2,k \rrbracket \cap U(D)}{\prod}w_j^{-1}
$$
matches with the upper shingle equation expressing $z_{k,lo}$ in terms of 
$z_{1,lo}$. Indeed, if $k$ is the underpass coming immediately after 
underpass $1$, Equation~\eqref{eq:uppershingle} says:
$$
z_{k,lo}=z_{1,lo}w_k^{-1}\underset{j \in \llbracket 2,k-1\rrbracket}{\prod}w_j.
$$
As crossings $2,3,\ldots k-1$ correspond to overpasses and $k$ to an 
underpass, we also have
$$
z_{k,k+1}=z_{1,lo} w_k^{-1} \underset{j \in \llbracket 2,k-1\rrbracket}{\prod}w_j.
$$
By induction, we find that $z_{k,k+1}=z_{k,lo}$ for any underpass $k$.

The second case is then a consequence of the lower triangle equation 
$z_{k+1,li}=\frac{1}{w_{k+1}z_{k+1,lo}}$, and the fact that 
$z_{k,k+1}=z_{k+1,k+2}w_{k+1}$ as $k+1$ is an underpass.

Note that $z_{2,ui}=\frac{z_{1,lo}}{w_\mu}$ by 
Equation~\eqref{eq:holonomy_meridian}, so the fourth case is valid for the 
arc $[1,2]$. Similarly to case 1, we can prove case 4 for other arcs ending 
in an overpass from the lower shingle equations by induction.

Finally, the third case follows as $z_{k,uo}=\frac{1}{w_k z_{k,ui}}$, and 
$z_{k,k+1}=w_k z_{k-1,k}$.
\end{proof}

In the rest of the paper, we will often use the arc parameters $z_{k,k+1}$ 
defined above to express equations in $\calG_D.$

For instance, thanks to Proposition \ref{prop.elimination}, we can rewrite 
the big region 
equations $r_k=1$ as equations $r_k(w)=1$, where $r_k(w)$ is expressed in 
terms of the variables $w$ only.

\begin{remark}
\label{remark.arcparam}
Although the arc parameters $z_{l,l+1}$ are just monomials in the $w$ 
variables, they are helpful for writing down the equations defining 
$\calG_D$ in a more compact way. When the choice of a crossing $c$ is 
implicit, 
we introduce a simplified notation for the parameters associated to arcs 
neighboring $c$. We will write $z_a, z_b, z_{a'},z_{b'}$ for the parameters 
associated to the inward half of the overpass, inward half of underpass, 
outward half of underpass and outward half of underpass.

With this convention, at any crossing we have:
$$
z_{ui}=\frac{z_a}{w_\mu}, \ z_{li}=\frac{1}{z_b}, 
\ z_{uo}=\frac{w_\mu}{z_{a'}}, \ \textrm{and} \ z_{lo}=z_{b'}.
$$
\end{remark}

For instance, we get a new expression of the holonomy of the longitude:

\begin{proposition}
With the convention of Remark \ref{remark.arcparam}, the holonomy of the 
zero-winding number longitude is expressed by:
\begin{equation}
\label{eq:hollongitude2}w_\lambda=w_\mu^{-wr(D)}\underset{X(D)}{\prod} w 
\left(\frac{1-\frac{w_\mu}{z_{a'}}}{1-\frac{w_\mu}{z_a}}\right) 
\left(\frac{1-z_{b'}}{1-z_b}\right) \,.
\end{equation}
\end{proposition}

\begin{proof}
By Equation~\eqref{eq:hollongitude} we have:
\begin{eqnarray*} 
&w_\lambda=w_\mu^{-wr(D)}\underset{X(D)}{\prod} 
\frac{z_{uo}''z_{ui}'}{z_{lo}'z_{li}''}
\\ &=w_\mu^{-wr(D)}\underset{X(D)}{\prod} \left(\frac{1-\frac{z_{a'}}{w_\mu}}{1
-\frac{z_a}{w_\mu}}\right) \left( \frac{1-z_{b'}}{1-z_b}\right)
\\ &=w_\mu^{-wr(D)}\underset{X(D)}{\prod}\frac{z_{a'}}{z_a} 
\left(\frac{1-\frac{w_\mu}{z_{a'}}}{1-\frac{w_\mu}{z_a}}\right) 
\left( \frac{1-z_{b'}}{1-z_b}\right)
\\ &=w_\mu^{-wr(D)}\underset{X(D)}{\prod}w 
\left(\frac{1-\frac{w_\mu}{z_{a'}}}{1-\frac{w_\mu}{z_a}}\right) 
\left( \frac{1-z_{b'}}{1-z_b}\right) \,.
\end{eqnarray*} 
\end{proof}

\subsection{Analysis of big region equations}
\label{sub.loop}

Recall that the big region equations are parametrized by the
regions of the planar diagram $D$, i.e., by the connected components of
$S^2\setminus D$. In this section, we give an alternative set of equations 
which are parametrized by the crossings of $D$, and we call those the
loop equations. 

Our motivation comes from the fact that we will later match
the loop equations with equations that come from a state sum formula for the
colored Jones polynomial.

Consider a crossing $c$ in the labeled diagram $D$. Recall from 
Section~\ref{sub.labelD} that
$c$ has two labels $j<j'$. The arc $[j,j']$ starts and ends at the same 
crossing, hence one may close it up to obtain a loop $\gamma_c$.
For a region $R_i$ of the diagram, let us pick a point $p_i$ in the 
interior of $R_i$. We write $w(\gamma_c,p_i)$ for the winding number of 
$\gamma$ relative to the point $p_i$. The big region equation corresponding 
to the region $R_i$ is $r_i=1$, where $r_i$ is the product of corners 
factors, see Equation~\eqref{eq:bigregion} and Figure~\ref{fig:cornerfactors}. 
The loop equation $L_c=1$ is then defined by
\begin{equation}
\label{eq:loopequation}
L_c=\underset{R_i \ \textrm{region}}{\prod} r_i^{w(\gamma_c,p_i)} \,.
\end{equation}
We also introduce
\be
\label{eq.L0}
L_0=\underset{R_i \ \textrm{region of } D}{\prod} r_i^{w(K,p_i)} \,.
\ee

\begin{proposition}
\label{prop:basis} 
The set of equations $L_0=1$, $L_{c}=1$ for all $c \in X(D)$ is equivalent 
to the set of equations $r_i=1$ for all region $R_i$ of $D$. 
\end{proposition}

\begin{proof}
The equations $L_0=1$, $L_c=1$ are clearly implied by 
the big region equations $r_i=1$ as the $L_c$'s and $L_0$ are monomials 
in the $r_i$'s. We will show that the $r_i$'s are also monomials in $L_0$ 
and the $L_c$'s, and thus equations $r_i=1$ are a consequence of 
loop equations.

Let us consider the diagram $D$ as an oriented $4$-valent graph embedded 
in $S^2$. For any $\delta\in H_1(D,\Z)$, we can also introduce 
a loop equation
$$
L_{\delta}=\underset{R_i \ \textrm{region}}{\prod} r_i^{w(\delta,p_i)}.
$$
Note that $\delta \rightarrow L_{\delta}$ is a morphism of group 
$H_1(D,\Z) \rightarrow \C^*$ and that the equation $r_i$ can be presented 
in this form too: 

Indeed, chose $\delta=\partial R_i$ with positive orientation. Then 
$w(\delta,p_j)=0$ if $j\neq i$, and $w(\delta,p_i)=1$, hence $L_{\delta}=r_i$.

Thus we only need to prove that $H_1(D,\Z)$ is generated by $K$ and the 
classes $\gamma_c$. The diagram $D$ has $c(D)$ vertices and 
$2c(D)$ edges, and thus $H_1(D,\Z)=\Z^{c(D)+1}$. So we need to show that $K$ 
and the loops $\gamma_c$ are a $\Z$-basis of $H_1(D,\Z)$. To do this 
we first show that they are linearly independent in the space of $1$-chains 
$C_1(D,\Z)$.

Recall that we fixed a labeling of overpasses and underpasses in 
$[ 1, 2c(D) ]$ following the knot $K$. Note that the arcs 
$[1,2], [2,3], \ldots [2c(D),1]$ give a basis of $C_1(D,\Z)$. We order this 
basis with the convention $[1,2]<[2,3]<\ldots <[2c(D),1]$.

Then $K=[1,2]+[2,3]+\ldots +[2c(D),1]$ in $C_1(D,\Z)$, and if a crossing $c$ 
has labels $j<j'$, then $\gamma_c=[j,j+1]+\ldots +[j'-1,j']$.

We see that $K$ is not in the space generated by the $\gamma_c$ as it 
is the only one with non-zero coordinate along $[2c(D),1]$.

Moreover, the loops $\gamma_c$ are linearly independent as the indices 
of their first non-zero coordinates are all different.

So $K$ and the $\gamma_c$ are linearly independent in $H_1(D,\Z)$, and 
thus a $\QQ$-basis of $H_1(D,\QQ)$. We can actually show that they form a 
$\Z$-basis of $H_1(D,\Z)$. Indeed if $\delta \in H_1(D,\Z)$, we can 
subtract a $\Z$-linear combination of $K$ and the $\gamma_c$'s to 
$\delta$ to obtain an element with $0$ coordinate on $[2c(D),1]$ and each 
$[j,j+1]$ for each crossing with labels $j<j'$. This element has then to 
be zero as $(K,\gamma_c)$ is a $\QQ$-basis of $H_1(D,\QQ)$. 

Thus $K$ and the $\gamma_c$'s  generate $H_1(D,\Z)$, and the $r_i$'s 
are monomials in the $L_0,L_c$. 
\end{proof}

\subsection{Formulas for the loop equations}
\label{sub.floops}

In this section, we simplify the equations $L_0,L_c$ which we defined as 
monomials in the big region equations. Our goal is to express those 
equations in terms of the arc parameters $z_{k,k+1}$ introduced in 
Section~\ref{sub.red1}, which we recall are monomials in the $w$ variables.

\begin{proposition}
\label{prop:loopeq}
Let $c$ be a crossing of $D$ with labels $j<j'$.
For $k\in [ j , j' ]$, let $\varepsilon(k)=1$ if $k$ 
corresponds to a positive crossing and $\varepsilon(k)=-1$ otherwise. Let 
also $u_+(k)=\frac{1+\varepsilon(k)}{2}$ and 
$u_-(k)=\frac{1-\varepsilon(k)}{2}$. Then we have:
\begin{align}
\label{eq:loopequation2}
L_c &=K_c \underset{k \in \llbracket j+1,j'-1 \rrbracket \cap O(D)}{\prod}
\left(\frac{z_b^{u_-(k)}}{z_{b'}^{u_+(k)}}\right)
\left( \frac{1-\frac{w_\mu}{z_{a'}}}{1-\frac{w_\mu}{z_a}}\right) 
\\ \notag & \quad
\times \underset{k \in \llbracket j+1,j'-1 \rrbracket \cap 
U(D)}{\prod}w_\mu^{\varepsilon (k)}\left(\frac{z_a^{u_-(k)}}{z_{a'}^{u_+(k)}}\right)
\left(\frac{1-z_b}{1-z_{b'}}\right),
\end{align}
where in the above we set 
$$
K_c= \begin{cases}
\left(\frac{1}{z_{b_c'}}\right)
\frac{\left(1-\frac{w_\mu}{z_{a_c'}}\right)\left(1-z_{b_c}\right)}{(1-w_c)}
& \text{if $j$ is an overpass and $\varepsilon (c)=+1$}, \\
\left(-\frac{z_{a_c'}}{w_\mu}\right)
\frac{\left(1-\frac{w_\mu}{z_{a_c'}}\right)(1-z_{b_c})}{(1-w_c)}
& \text{if $j$ is an overpass and $\varepsilon (c)=-1$},
\\
\left(\frac{w_\mu}{z_{a_c'}}\right)\frac{(1-w_c)}{
\left(1-\frac{w_\mu}{z_{a_c}}\right)\left(1-z_{b_c'}\right)}
& \text{if $j$ is an underpass and $\varepsilon (c)=+1$},
\\
(-z_{b_c'})\frac{(1-w_c)}{\left(1-\frac{w_\mu}{z_{a_c}}\right)(1-z_{b_c'})}
& \text{if $j$ is an underpass and $\varepsilon (c)=-1$}.
\end{cases}
$$
\end{proposition}

\begin{proof}

\begin{figure}[!htpb]
\centering
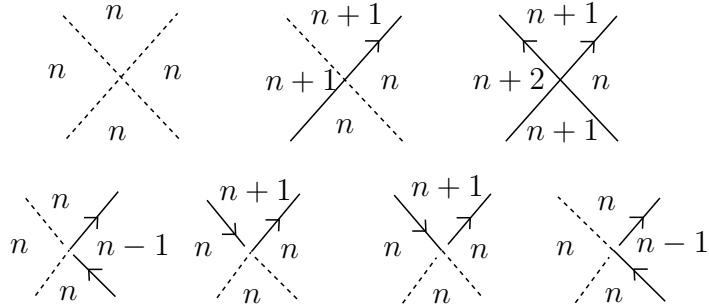
\caption{The local pattern of winding numbers near a crossing. Strands 
of $\gamma$ are represented by solid lines, strands of $\gamma'$ by 
dashed lines. The bottom row corresponds to the $4$ different possibilities 
for over-/underpass $j$: positive overpass, negative overpass, positive 
underpass,  or negative underpass.}
\label{fig:windingnumbers}
\end{figure}

We recall that $\gamma$ is the loop obtained from the arc $[j,j']$ of $D$ 
by gluing its two ends together. Let also $\gamma'$ be the complementary 
loop of $\gamma$, which is obtained from the arc $[j',j]$  by gluing the 
two ends. Note that $\gamma'$ goes through the underpass labeled $1$.  

As $L_c=\underset{R_i \ \textrm{region}}{\prod} r_i^{w(\gamma,p_i)}$ is a 
product of big region equations, and each big region factor is a product 
of corner factors, we can rewrite $L_c$ as a product of corner factors. 
Each corner $v$ of $D$ appears in one region $R_i$ only , and the winding 
number $w(\gamma,v)$ of $\gamma$ around $v$ is the same as $w(\gamma,p_i)$. 
Thus we may rewrite $L_c$ as
$$
L_c=\underset{v \ \textrm{corner of }D}{\prod} f(v)^{w(\gamma, v)},
$$ 
where the corner factors $f(v)$ are those of Figure~\ref{fig:cornerfactors}.

Figure~\ref{fig:windingnumbers} shows the local pattern of winding numbers 
of corners near a crossing of $D$, depending which neighboring arcs belong 
to $\gamma$ and $\gamma'$. First let us note for a crossing between two 
strands of $\gamma'$, all local winding numbers are equal, thus the 
crossing contributes by the product of all $4$ corners factors to some power. 
However, at any positive crossing, the product of corner factors is 
$$
(w'z_{uo}''z_{lo}'')(w''z_{uo}'z_{li}')(w'z_{ui}''z_{li}'')(w''z_{ui}'z_{lo}')
=\frac{1}{w^2 z_{ui}z_{uo}z_{li}z_{lo}}=1
$$
by the rule $z z' z''=-1$ and the triangle equations. Similarly, at any 
positive crossing, the product of corner factors is 
$$
(w''z_{uo}'z_{lo}')(w'z_{uo}''z_{li}'')(w''z_{ui}'z_{li}')(w'z_{ui}''z_{lo}'')
=\frac{1}{w^2 z_{ui}z_{uo}z_{li}z_{lo}}=1.
$$
So crossings between two strands of $\gamma'$ do not contribute to $L_c$.

Next we consider a crossing between one strand of $\gamma$ and one 
strand of $\gamma'$. By the local winding numbers shown in 
Figure~\ref{fig:windingnumbers} and that fact that the product of the 
$4$ corner factors at a crossing is $1$, such a crossing contributes by 
the product of the two corner factors to the left of $\gamma$. Similarly, 
for a crossing between two strands of $\gamma$, we get the product of the 
two corner factors to the left of the first strand times the two corner 
factors to the left of the other strand.

Hence, each overpass or underpass $l \in \llbracket j+1 , j'-1 \rrbracket$ 
of $\gamma$ contributes to one factor $K_l$ which is the product of the two 
left corner factors. By the rule described in Figure~\ref{fig:cornerfactors}, 
for a positive overpass we get 
\begin{align*}
K_l &=(w'z_{uo}''z_{lo}'')(w''z_{ui}'z_{lo}')=\frac{z_{ui}'z_{uo}''}{wz_{lo}}
=z_{li}z_{ui}'z_{uo}''
\\ &=
\frac{1}{z_b}\left(\frac{1-\frac{z_{a'}}{w_\mu}}{1-\frac{z_a}{w_\mu}}\right)
= \frac{z_{a'}}{z_a z_b}\left(\frac{1-\frac{w_\mu}{z_{a'}}}{1-\frac{w_\mu}{z_a}}
\right)
=\frac{1}{z_{b'}}\left(\frac{1-\frac{w_\mu}{z_{a'}}}{1-\frac{w_\mu}{z_a}}\right),
\end{align*}
where the last equality comes from the fact that, at any crossing, 
$\frac{z_{a'}}{z_a}=\frac{z_b}{z_{b'}}=w$. Similarly, at a negative overpass 
we get:
\begin{multline*}
K_l=(w'z_{uo}''z_{li}'')(w''z_{ui}'z_{li}')=\frac{z_{uo}''z_{ui}'}{wz_{li}}
=z_{lo}z_{uo}'z_{ui}''
\\ =z_{b'}\left( \frac{1-\frac{z_{a'}}{w_\mu}}{1-\frac{z_a}{w_\mu}}\right) 
=\frac{z_{b'}z_{a'}}{z_a}\left(\frac{1-\frac{w_\mu}{z_{a'}}}{
1-\frac{w_\mu}{z_a}}\right)=z_b \left(\frac{1-\frac{w_\mu}{z_{a'}}}{
1-\frac{w_\mu}{z_a}}\right).
\end{multline*}
At a positive underpass we get:
$$
K_l=(w''z_{ui}'z_{lo}')(w'z_{ui}''z_{li}'')=\frac{z_{lo}'z_{li}''}{wz_{ui}}
=z_{uo}z_{lo}'z_{li}''= \frac{w_\mu}{z_{a'}}\left(\frac{1-z_b}{1-z_{b'}}\right),
$$
and, finally, at a negative underpass we get:
$$
K_l=(w'z_{uo}''z_{li}'')(w''z_{uo}'z_{lo}')=\frac{z_{lo}'z_{li}''}{wz_{uo}}
=z_{ui}z_{lo}'z_{li}''= \frac{z_a}{w_\mu}\left(\frac{1-z_b}{1-z_{b'}}\right).
$$
All those overpass/underpass factors correspond to the ones in 
Equation~\eqref{eq:loopequation2}. Finally we turn to the contribution $K_c$ 
of crossing $c$. By the local pattern of winding numbers in 
Figure~\ref{fig:windingnumbers}, and the corner factors rule of 
Figure~\ref{fig:cornerfactors}, we have, if $j$ is a positive overpass:
$$
K_c=\frac{1}{w''z_{uo}'z_{li}'}=\frac{
(1-\frac{w_\mu}{z_{a'}})(1-\frac{1}{z_b})}{(1-\frac{1}{w})}=
\left(\frac{w}{z_b}\right)\frac{(1-\frac{w_\mu}{z_{a'}})(1-z_b)}{(1-w)}=
\left( \frac{1}{z_{b'}} \right) \frac{(1-\frac{w_\mu}{z_{a'}})(1-z_b)}{(1-w)}.
$$
If $j$ is a negative overpass, we have:
$$
K_c=w'z_{uo}''z_{li}''=\left(\frac{1}{1-w}\right) (1-\frac{z_{a'}}{w_\mu})(1-z_b)
=\left(-\frac{z_{a'}}{w_\mu}\right)\frac{(1-\frac{w_\mu}{z_{a'}})(1-z_b)}{(1-w)}.
$$
If $j$ is a positive underpass, then:
$$
K_c=w''z_{ui}'z_{lo}'=\frac{(1-\frac{1}{w})}{(1-\frac{z_a}{w_\mu})(1-z_{b'})}=
\left( \frac{w_\mu}{w z_a}\right) \frac{1-w}{(1-\frac{w_\mu}{z_a})(1-z_{b'})}=
\left( \frac{w_\mu}{z_{a'}}\right) \frac{1-w}{(1-\frac{w_\mu}{z_a})(1-z_{b'})}.
$$
Finally if $j$ is a negative underpass, then:
$$
K_c=\frac{1}{w'z_{ui}''z_{lo}''}=\frac{(1-w)}{
(1-\frac{w_\mu}{z_a})(1-\frac{1}{z_{b'}})}=(-z_{b'})\frac{(1-w)}{
(1-\frac{w_\mu}{z_a})(1-z_{b'})}.
$$
We clearly see that in each case the factor $K_c$ matches with that of 
Proposition~\ref{prop:loopeq}.
\end{proof}

We want to rearrange the loop equations slightly, grouping together the 
factors $w_\mu^{\varepsilon(k)}$ on the one side and the factors 
$\frac{z_b^{u_-(k)}}{z_{b'}^{u_+(k)}}$ and $\frac{z_{a}^{u_-(k)}}{z_{a'}^{u_+(k)}}$ 
on the other side. For the former we claim:

\begin{lemma}
\label{prop:writhe}
Let $c$ be a crossing of $D$ with labels $j<j'$, $\gamma$ the loop 
$[j,j']/_{j=j'}$, and $\gamma'$ the loop $[j',j]/_{j=j'}$. For $l\in (j,j')$ 
an over- or underpass, let $\varepsilon(l)$ be the sign of the corresponding 
crossing. Then we have
$$
\underset{l \in U(D) \cap (j,j')}{\sum} \varepsilon(l)=
\underset{l \in (j,j')}{\sum}\frac{\varepsilon(l)}{2}=
\mathrm{wr}(\gamma)+\mathrm{lk}(\gamma,\gamma').
$$
\end{lemma}

\begin{remark}
By the above lemma, the factors $w_\mu^{\varepsilon(k)}$ in the product on the 
right of Equation~\eqref{eq:loopequation2} group up to one factor 
$w_\mu^{\mathrm{wr}(\gamma)+\mathrm{lk}(\gamma,\gamma')}$.
\end{remark}

\begin{proof}
The crossings of $D$ that are in $(j,j')$ are of two types: self-crossings 
of $\gamma$ and crossing between $\gamma$ and $\gamma'$. Self-crossings of 
$\gamma$ belong to both an overpass and an underpass $l\in (j,j')$, hence 
in both sums in the lemma, those crossings contribute to 
$c_+(\gamma)-c_-(\gamma)=\mathrm{wr}(\gamma)$. 

Moreover the linking number of $\gamma$ and $\gamma'$ can be computed 
in two ways as $\underset{l \in \gamma \cap \gamma'}{\sum} 
\frac{\varepsilon(l)}{2}$ or as $\underset{l \in \gamma 
\cap \gamma' \cap U(D)}{\sum} \varepsilon(l)$. Thus hence in both sums 
in the lemma mixed crossings contribute to $\mathrm{lk}(\gamma,\gamma')$.
\end{proof}

\begin{lemma}
\label{prop:conservation_of_color}
Let $c$ be a crossing of $D$ with labels $j<j'$. Then:
$$
\underset{k \in (j,j') \cap O(D)}{\prod}\frac{z_b^{u_-(k)}}{z_{b'}^{u_+(k)}} 
\underset{k \in (j,j') \cap U(D)}{\prod}\frac{z_{a}^{u_-(k)}}{z_{a'}^{u_+(k)}}
=C \underset{k \in (j,j') \cap O(D)}{\prod} 
(z_b z_{b'})^{-\frac{\varepsilon(k)}{2}}\underset{k \in (j,j') 
\cap U(D)}{\prod}(z_a z_{a'})^{-\frac{\varepsilon(k)}{2}},
$$
where $C=\left(\frac{z_{b_c}}{z_{a_c'}} \right)^{\frac{1}{2}}=
\left(\frac{z_{b_c'}}{z_{a_c}} \right)^{\frac{1}{2}}$ if $j$ is an overpass 
and $C=\left( \frac{z_{a_c}}{z_{b_c'}}\right)^{\frac{1}{2}}=
\left(\frac{z_{a_c'}}{z_{b_c}} \right)^{\frac{1}{2}}$.
\end{lemma}

\begin{proof}
We have by definition of $u_+(k)$ and $u_-(k)$: 
$$
\frac{z_b^{u_-(k)}}{z_{b'}^{u_+(k)}}=\left(\frac{z_b}{z_{b'}} 
\right)^{\frac{1}{2}}(z_b z_b')^{-\frac{\varepsilon(k)}{2}}, \ \textrm{and} 
\ \frac{z_a^{u_-(k)}}{z_{a'}^{u_+(k)}}=\left(\frac{z_a}{z_{a'}} 
\right)^{\frac{1}{2}}(z_a z_a')^{-\frac{\varepsilon(k)}{2}}.
$$
Moreover, as at any crossing $\frac{z_b}{z_{b'}}=\frac{z_{a'}}{z_a}=w$, we have:
\begin{align*}
\underset{k \in (j,j') \cap O(D)}{\prod}\left(\frac{z_b}{z_{b'}} 
\right)^{\frac{1}{2}}\underset{k \in (j,j') \cap 
U(D)}{\prod}\left(\frac{z_a}{z_{a'}} \right)^{\frac{1}{2}} &=
\underset{k \in (j,j') \cap O(D)}{\prod}\left(\frac{z_{a'}}{z_a} 
\right)^{\frac{1}{2}}\underset{k \in (j,j') \cap U(D)}{\prod}
\left(\frac{z_{b'}}{z_b} \right)^{\frac{1}{2}}
\\ 
&=\underset{k \in (j,j')}{\prod}\left(\frac{z_{k,k+1}}{z_{k-1,k}}
\right)^{\frac{1}{2}}=\left(\frac{z_{j'-1,j'}}{z_{j,j+1}}\right)^{\frac{1}{2}}.
\end{align*}
Finally, if $j$ is an overpass then 
$\left(\frac{z_{j'-1,j'}}{z_{j,j+1}}\right)^{\frac{1}{2}}=
\left(\frac{z_{b_c}}{z_{a_c'}}\right)^{\frac{1}{2}}$ as $z_{a_c'}=z_{j,j+1}$ and 
$z_{b_c}=z_{j'-1,j'}$. Similarly, 
$\left(\frac{z_{j'-1,j'}}{z_{j,j+1}}\right)^{\frac{1}{2}}=
\left(\frac{z_{a_c}}{z_{b_c'}}\right)^{\frac{1}{2}}$ if $j$ is an underpass.
\end{proof}
From Proposition~\ref{prop:loopeq} together with Lemma~\ref{prop:writhe} 
and~\ref{prop:conservation_of_color}, we obtain another formula for the 
loop equation:

\begin{proposition}
\label{prop:loopeq2}
Let $c$ be a crossing of $D$ with labels $j<j'$ and let $L_c$ be the 
associated loop equation. If $k \in (j,j')$, let $\varepsilon(k)$ be 
the sign of the corresponding crossing. Then:
\begin{equation}
\label{eq:loopeq3}
L_c=K_c' \underset{k \in (j,j')  \cap O(D)}{
\prod}\left(\frac{w_\mu}{z_b z_{b'}}\right)^{\frac{\varepsilon(k)}{2}}
\left( \frac{1-\frac{w_\mu}{z_{a'}}}{1-\frac{w_\mu}{z_a}}\right) 
\\ \times \underset{k \in (j,j')\cap U(D)}{\prod}
\left(\frac{w_\mu}{z_a z_{a'}}\right)^{\frac{\varepsilon(k)}{2}}
\left(\frac{1-z_b}{1-z_{b'}}\right)
\end{equation}
where $K_c'$ is obtained from $K_c$ of Proposition~\ref{prop:loopeq} by 
replacing respectively a factor $\left(\frac{1}{z_{b_c'}}\right)$,  
$z_{a_c'}$, $\left( \frac{1}{z_{a_c'}}\right) $, or $z_{b_c'}$ by 
$\frac{1}{(z_{a_c}z_{b_c'})^{\frac{1}{2}}}$,  $(z_{a_c'}z_{b_c})^{\frac{1}{2}}$, 
$\frac{1}{(z_{a_c'}z_{b_c})^{\frac{1}{2}}} $, or $(z_{a_c} z_{b_c'})^{\frac{1}{2}}$ 
if $j$ is a positive overpass, a negative overpass, a positive underpass 
or a negative underpass.
\end{proposition}

Finally, we turn to the expression of the last loop equation 
$L_0=\underset{R_i\ \textrm{region}}{\prod}r_i^{w(K,R_i)}$ that we introduced 
in Section~\ref{sub.loop}.

\begin{proposition}
\label{prop:loopeqK} 
We have the formula:
$$
L_0=\underset{c \in X(D)}{\prod}
\left( \frac{w_\mu}{z_a z_b}\right)^{\varepsilon(c)} 
\frac{(1-\frac{w_\mu}{z_{a'}})(1-z_b)}{(1-\frac{w_\mu}{z_a})(1-z_{b'})}
$$
\end{proposition}

\begin{proof}
We proceed similarly as in the proof of~\ref{prop:loopeq}. As we are taking 
the whole knot $K$ instead of one of the loops $\gamma_c$, the local 
pattern of winding numbers at any crossing looks like the third drawing in 
Figure~\ref{fig:windingnumbers}.
\\ By the corner factor rule of Figure~\ref{fig:cornerfactors}, we get 
a factor 
$$
\frac{z_{ui}'z_{lo}'}{z_{uo}'z_{li}'}=
\frac{(1-\frac{w_\mu}{z_{a'}})(1-\frac{1}{z_b})}{(1-\frac{z_a}{w_\mu})(1-z_{b'})}
=
\left( \frac{w_\mu}{z_a z_b}\right)\frac{(1-\frac{w_\mu}{z_{a'}})(1-z_b)}{
(1-\frac{w_\mu}{z_a})(1-z_{b'})}
$$
at a positive crossing and a factor:
$$
\frac{z_{uo}''z_{li}''}{z_{ui}''z_{lo}''}=
\frac{(1-\frac{z_{a'}}{w_\mu})(1-z_b)}{(1-\frac{w_\mu}{z_a})(1-\frac{1}{z_{b'}})}
=
\left( \frac{z_{a'}z_{b'}}{w_\mu}\right)\frac{(1-\frac{w_\mu}{z_{a'}})(1-z_b)}{
(1-\frac{w_\mu}{z_a})(1-z_{b'})}=\left( \frac{z_a z_b }{w_\mu}\right)\frac{
(1-\frac{w_\mu}{z_{a'}})(1-z_b)}{(1-\frac{w_\mu}{z_a})(1-z_{b'})},
$$
at a negative crossing, using that $z_a z_b=z_{a'}z_{b'}$ at any crossing.
\end{proof}

\subsection{A square root of the holonomy of the longitude}
\label{sub.sroot}

In this section, we show that the holonomy of the longitude 
$w_{\lambda}$ admits a square root in $\C[\calG_D]$. We prove the following.

\begin{proposition}
\label{prop:sqroot}Let $s$ be defined by
\begin{equation}
\label{eq:sqroot}
s=\underset{X(D)}{\prod} 
\frac{(1-\frac{w_{\mu}}{z_a})}{(1-\frac{w_{\mu}}{z_{a'}})}
w^{-1/2}(z_a z_b)^{\frac{\varepsilon(c)}{2}}.
\end{equation}
Then $s \in \C(w_{\mu},w_0,w_c)$ and $s^2=\frac{1}{w_{\lambda}L_0}$.
\end{proposition}

\begin{proof}
By Equation~\eqref{eq:hollongitude2},
$$
w_{\lambda}=\underset{X(D)}{\prod}w_{\mu}^{-\varepsilon(c)} w 
\left( \frac{1-\frac{w_{\mu}}{z_{a'}}}{1-\frac{w_{\mu}}{z_a}}\right)
\left(\frac{1-z_{b'}}{1-z_b}\right) \,,
$$
and by Equation~\eqref{eq:loopeq3}:
$$
L_0=\underset{X(D)}{\prod} 
\left( \frac{w_{\mu}}{z_a z_b}\right)^{\varepsilon(c)} 
\left( \frac{1-\frac{w_{\mu}}{z_{a'}}}{1-\frac{w_{\mu}}{z_a}}\right) 
\left( \frac{1-z_b}{1-z_{b'}}\right) \,.
$$
Those two equations clearly imply that $s^2=\frac{1}{w_{\lambda}L_0}$. The 
non-trivial part is to show that $s$ is actually in $\C(w_{\mu},w_0,w_c)$, 
which is equivalent to showing the degree of the monomial 
$\underset{X(D)}{\prod} w z_a z_b$ is even in each of the variable 
$w_{\mu},w_0$ and $w_c$.

First we note that all arc parameters $z_a, z_b$ have degree $0$ along 
$w_{\mu}$ and degree $1$ along $w_0$. So what we need to show is that the 
product $\underset{X(D)}{\prod} z_a z_b$ has odd degree along each variable 
$w_c$ associated to a crossing. We remark that this product is also the 
product of all arc parameters as each arc is an inward arc of exactly one 
crossing.

Let $c$ be a crossing with labels $j<j'$. Then for any arc $[k,k+1]$ the 
arc parameter $z_{k,k+1}$ is of the form 
$z_{k,k+1}=w_0 w_c^{\varepsilon} 
\underset{c'\neq c}{\prod}w_{c'}^{\varepsilon_{c'}}$, 
where $\varepsilon \in \lbrace -1,0,1\rbrace$, and $\varepsilon \neq 0$ if 
and only if $[k,k+1]\subset [j,j']$. So all we have to show is that $j'-j$ 
is always odd for any crossing $c$. The reason is that the loop 
$\gamma=[j,j']/_{j\sim j'}$ has $j'-j-1$ intersection points with the rest of 
$K$, and those intersection points bound a collection of segments, which 
are the intersection of $K$ with a disk bounded by $\gamma$. So $j'-j-1$ is 
always even. 
\end{proof}


\section{$q$-holonomic functions, creative telescoping and 
certificates}
\label{sec.qholo}

In this section we recall some properties of $q$-holonomic functions,
creative telescoping and certificates, which we will combine with
a state sum formula for the colored Jones polynomial to prove our
main Theorem~\ref{thm.1}. Recall that a $q$-holonomic function 
$f: \BZ \to \BQ(q)$ is one that satisfies a non-zero recursion relation
of the form~\eqref{eq.Jrec}, i.e., a function with 
annihilator~\eqref{eq.annf} satisfying $\Ann(f) \neq 0$. $q$-holonomic 
functions of several variables are defined using a notion of Hilbert series
dimension, and are closed under sums, products as well as summation of 
some of their variables. Building blocks of $q$-holonomic functions are
the proper $q$-hypergeometric functions of~\cite{WZ}. For a detailed
discussion of $q$-holonomic functions, we refer the reader to the survey
article~\cite{GL:survey}.

The following proposition is the fundamental theorem of $q$-holonomic 
functions. When $F$ is proper $q$-hypergeometric, a proof was given in 
Wilf-Zeilberger~\cite{WZ}.
A detailed proof of the next proposition, as well as a self-contained 
introduction to $q$-holonomic functions, we refer the reader 
to~\cite{GL:survey}. 

\begin{proposition}
\label{prop:qholo}
\rm{(a)}
Proper $q$-hypergeometric functions are $q$-holonomic. \newline
\rm{(b)}
Let $F: \BZ^{r+1} \to \BQ(q)$ be $q$-holonomic in the variables $(n,k) 
\in \BZ \times \BZ^r$ such that $F(n,\cdot)$ has finite support for any 
$n$ and let $f: \BZ \to \BQ(q)$ be defined by
$$
f(n)=\underset{k \in \Z^r}{\sum} F(n,k).
$$
Then $f$ is $q$-holonomic.
\end{proposition}

The above proposition combined with an $R$-matrix state-sum formula
for the colored Jones polynomial implies that the colored Jones polynomial
of a knot (or link, colored by representations of a fixed simple Lie
algebra) is $q$-holonomic~\cite{gale2}.

With the notation of the above proposition, a natural question is 
how to compute $\Ann(f)$ given $\Ann(F)$. This is a difficult problem
practically unsolved. However, an easier question can be solved: namely
given $\Ann(F)$, how to compute a nonzero element in $\Ann(f)$. The
answer to this question is given by certificates, which are synonymous to 
the method of creative telescoping, coined by Zeilberger~\cite{Zeil:creative}.
The latter aims at computing recursions for holonomic functions obtained 
by summing/integrating all but one variables. For a detailed discussion 
and applications, see~\cite{AB,WZ} and also \cite{BLS}. 

\begin{proposition}
\label{prop:cert}
{\rm (a)} Let $F$ and $f$ be as in Proposition~\ref{prop:qholo}, and
consider the map $\varphi$ from~\eqref{eq.phi}. Let 
\be
\label{eq.PF}
P \in \Ann(F)\cap \BQ[q,Q] \langle E, E_i \rangle \,.
\ee
Then $\varphi(P) \in \Ann(f)$. 
\newline
{\rm (b)} There exist $P$ as above with $\varphi(P) \neq 0$. 
\end{proposition}

Nonzero elements $P$ as in~\eqref{eq.PF}
are called ``certificates'', and those that satisfy $\varphi(P) \neq 0$
are called ``good certificates''. Certificates are usually computed in
the intersection $\Ann(F) \cap \BQ(q,Q) \langle E, E_i \rangle$, where
membership reduces to a linear algebra question over the field $\BQ(q,Q)$
and then lifted to the ring $\BQ[q,Q] \langle E, E_i \rangle$ by clearing 
denominators.

Part (b) is shown in Zeilberger~\cite{Zeil:holo} and in detail in 
Koutschan's thesis~\cite[Thm.2.7]{Koutschan}. In the latter reference,
this is called the ``elimination property'' of holonomic ideals.
Part (a) is easy and motivates the name ``creative telescoping''. Indeed, 
one may write
$$
P(E,Q,E_i)=\tilde{P}(E,Q)+\underset{i=1}{\overset{d}{\sum}} 
(E_i-1)R_i(E,Q,E_i).
$$
A recurrence relation of this form is also called a certificate. After 
expanding the sum $\underset{k \in \Z^d}{\sum} P(E,Q,E_i)F(n,k)=0$, the terms
$$
\underset{k \in \Z^d}{\sum} (E_i-1)R_i(E,Q,E_i)F(n,k),
$$
are telescoping sums and thus equal to $0$. Finally, note that when
$F$ is proper $q$-hypergeometric, an operator $P$ as above may be found
by using its monomials as unknowns and solving a system of linear equations
of $PF/F$. Hence, once $P$ is found (and that is the difficult part), it is
easy to check that it satisfies the relation $PF=0$, which reduces to
an identity in a field of finitely many variables--hence the 
name ``certificate''.

Part (b) follows by multiplying an element of $\Ann(F)$ on the left if
necessary by a monomial in $Q_i$. We thank C. Koutschan for pointing this 
out to us.


\section{The colored Jones polynomial of a knot}
\label{sec.CJ}

\subsection{State sum formula for the colored Jones polynomial
of a knot diagram}

In this section, we use a diagram $D$ of an oriented knot $K$ to give
a (state sum) formula for the $n$-th colored Jones polynomial 
$J_K(n)\in \Z[q^{\pm 1}]$ of $K$. 
Such a formula is obtained by placing
an $R$-matrix at each crossing, coloring the arcs of the diagram with
integers, and contracting tensors as described for instance in Turaev's
book~\cite{Tu:book}. The formula described in this section follows the 
conventions introduced in \cite{gale1}; we also refer to \cite{gale1} 
for all proofs.

For $n\geqslant 0$, we define the $n$-th quantum factorial by 
$$
(q)_n=\underset{i=1}{\overset{n}{\prod}} (1-q^i).
$$
Note that quantum factorials satisfy the recurrence relation 
$(q)_{n+1}=(1-q^{n+1})(q)_n$ for any $n\geqslant 0$. As it will be helpful 
for us to have recurrence relations that are valid for any $n\in \Z$, we 
will use the following convention of quantum factorials and their inverses:
$$
(q)_n=\begin{cases} 
\prod_{j=1}^n  (1-q^i) & \textrm{if} \ n\geqslant 0,
\\ 0 & \textrm{if} \ n<0, 
\end{cases}
\qquad
\frac{1}{(q)_n}=\begin{cases}\frac{1}{
\underset{i=1}{\overset{n}{\prod}} (1-q^i)} & \textrm{if} \ n\geqslant 0,
\\ 0 & \textrm{if} \ n<0 \,.
\end{cases}
$$
With the above definition and with the notation of~\eqref{eq.EQ} we have: 
$$
(1-qQ)(E-(1-qQ)) \in \Ann((q)_n), \qquad 
((1-qQ)E-1) \in \Ann(1/(q)_n) \,.
$$
Fix a labeled diagram $D$ of an oriented knot $K$ as in 
Section~\ref{sub.labelD}. After possibly performing a local rotation, one can 
arrange $D$ so that at each crossing the two strands of $K$ are going upwards. 
The diagram $D$ is then composed of two types of pieces: 
the crossings (which can be possible or negative) and local extrema. Let 
$\mathrm{arc}(D)$ be the set of arcs of the diagram $D$, we say that a 
coloring 
$$
r : \mathrm{arc}(D) \longrightarrow \Z
$$
is $n$-admissible if the color of any arc is in $[ 0, n ]$ 
and for any crossing,  if $a,a',b,b'$ are the color of the neighboring arcs 
in shown in Figure~\ref{fig:CJPweights}, then $a'-a=b-b'=k\geqslant 0$. Let 
$S_{D,n}$ be the set of all $n$-admissible colorings of the arcs of $D$. 
Note that $S_{D,n}$ coincides with the set of lattice points in 
the $n$-th dilatation of a rational convex polytope $P_D$ 
defined by the $n$-admissibility conditions.

\begin{figure}[!htpb]
\centering
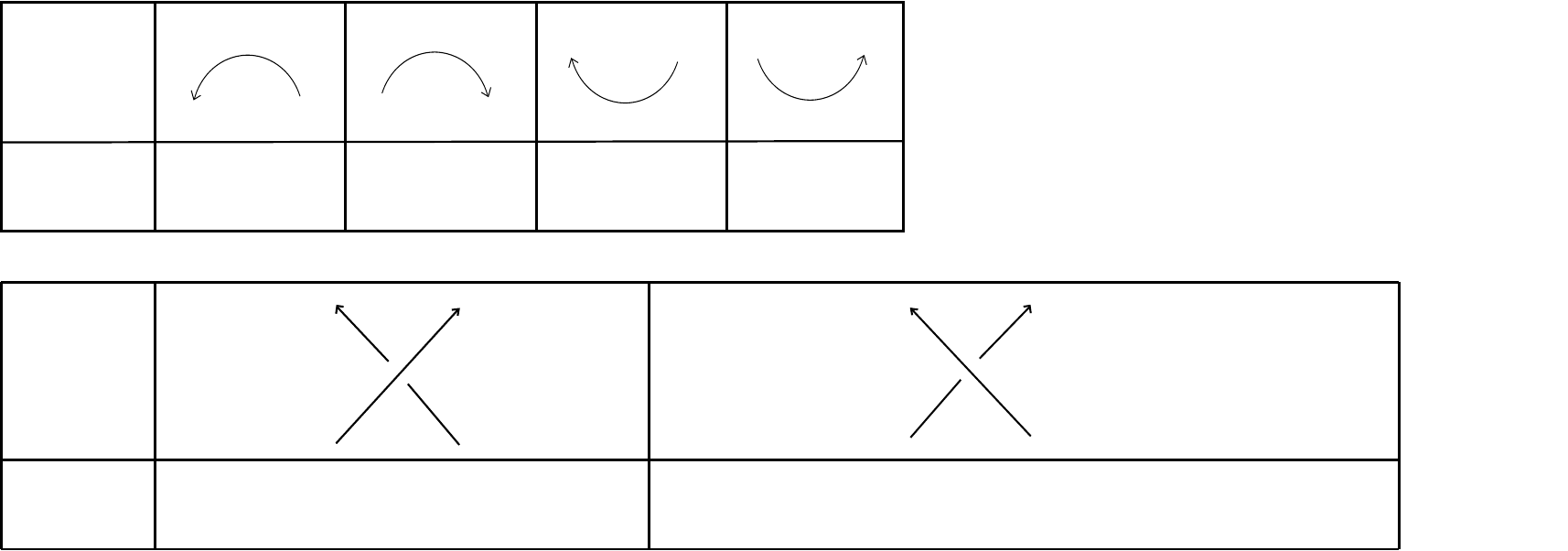
\caption{The local parts $X$ of $D$, their arc-colors $r$ and their 
weights $w(X,r)$. }
\label{fig:CJPweights}
\end{figure} 

For a proof of the next proposition, we refer to \cite[Sec.2]{gale1}. 

\begin{proposition}
The normalized $n$-th colored Jones polynomial of $K$ is obtained by 
the formula:
\be
\label{eq.Jw}
J_K(n)=q^{n/2} \underset{r\in S_{D,n}}{\sum}\ti w_D(n,r),
\ee
where $\ti w_D(n,r)=\underset{X \ \textrm{piece}}{\prod} w(X,r)$ 
is a product of weights associated to crossings and extrema of $D$ as 
shown in Figure~\ref{fig:CJPweights}.
\end{proposition}

The insertion of the factor $q^{n/2}$ in front of the above sum is done
for convenience only, so that $J_K(n)$ is a Laurent polynomial in $q$
rather than one in $q^{1/2}$. This normalization plays no role in the
AJ Conjecture. Note that we have $J_K(0)=1$ for every knot $K$ and 
$J_{\mathrm{Unknot}}(n)=\frac{1-q^{n+1}}{1-q}$ for any 
$n\geqslant 0$ and $J_K(1,q^{-1})/J_{\mathrm{Unknot}}(1,q^{-1})$ is the Jones 
polynomial of $K$. 

Note that the color of all arcs are completely determined by the 
shifts $(k_1,\dots,k_{c(D)}) \in \Z^{c(D)}$ associated to crossings and 
the color 
$k_0$ of the arc $[1,2]$. In other words, $r=r(k)$ is a linear function
of $k=(k_0,\dots,k_{c(D)}) \in \Z^{c(D)+1}$. 
Suppressing the dependence on $q$, we abbreviate $\ti w_D(n,r(k))(q)$ simply
by $w_D(n,k)$. 
  
When examining recurrence relations for the colored Jones it will be more 
convenient to express $J_K(n)$ as a sum over all $k\in \Z^{c(D)+1}$ rather 
than a sum over colorings $r$ in the set $S_{D,n}$ of lattice points in 
the rational convex polytope $P_D$. For this we have the lemma:

\begin{lemma}
For any knot $K$, we have:
\be
\label{eq.Jkn}
J_K(n)=\underset{k\in \Z^{c(D)+1}}{\sum} w_D(n,k) \,.
\ee
\end{lemma}

\begin{proof}
We recall that we have set the convention $\frac{1}{(q)_n}=0$ if $n<0$. 
From the definition of weights associated to crossings, we see that at 
any crossing the weight vanishes unless $k\geqslant 0$, $b'\geqslant 0$ 
and $a'\leqslant n$. 

Pick a coloring so that the associated weight is non-zero. Consider 
the color $c_{i,i+1}$ of the arc $[i,i+1]$. If $i$ is an underpass, then we 
get that $c_{i,i+1}\geqslant 0$. If on the other hand $i$ is an overpass, 
then $c_{i,i+1}=c_{i-1,i}+k_i$, so $c_{i,i+1}\geqslant c_{i-1,i}$. If $i-1$ is 
an underpass, one concludes that $c_{i,i+1}\geqslant c_{i-1,i} \geqslant 0$, 
else, one can continue until we meet an underpass $k$, and write 
$$
c_{i,i+1}\geqslant c_{i-1,i}\geqslant \ldots \geqslant c_{k,k+1}\geqslant 0.
$$
Thus if the weight is non-zero, the color of all arcs must be non-negative. 

Similarly, we can show that the color of all arcs muss be at most $n$. 
We already know that $c_{i,i+1}\leqslant n$ if $i$ is an overpass. Else, if 
$k$ is the overpass immediately before $i$, we have
$$
c_{i,i+1}\leqslant c_{i-1,i} \ldots \leqslant c_{k,k+1}\leqslant n.
$$
Thus any non-zero weight corresponds to an element of $S_{D,n}$.
\end{proof}

\subsection{The annihilator ideal of the summand of the 
state sum}

It is easy to see that the summand $w_D(n,k)$ of the state sum~\eqref{eq.Jw}
is a $q$-proper hypergeometric function in the sense of~\cite{WZ}. In this
section we compute generators of its annihilator ideal. To do so, we compute 
the effect of the shift operators $E$, $E_0$ and $E_c$ on $w_D(n,k)$.
Each operator is acting on exactly one of the $c(D)+2$ variables 
$(n,k)$ leaving all others fixed.

\begin{itemize}
\item
$E$ shifts $n$ to $n+1$.
\item
$E_0$ shifts $k_0$ to $k_0+1$. 
As the color of any other arc of $D$ is of the form 
$k_0+\underset{c\in X(D)}{\sum} \varepsilon_c k_c$ 
with $\varepsilon_c \in \lbrace -1,0,1\rbrace$, the operator $E_0$ 
actually shifts the color of all arcs up by $1$.
\item
$E_c$ for each crossing $c$ shifts $k_c$ to $k_c+1$.
\end{itemize}

The propositions of this section will match, after setting $q=1$, with
the gluing equations of the $5T$-spine of the knot projection.


Because we will later reduce our equations by plugging $q=1$, it will 
only matter to us that they are exact up to fixed powers of $q$. We will 
write $q^*$ for a power of $q$ which does not depend on $(n,k)$.

Let us start by considering the effect of $E_0$ on $w$.

\begin{proposition}
\label{prop:E_0-rec}
The summand $w_D(n,k)$ of the colored Jones polynomial satisfies:
\begin{equation}
\label{eq:E_0-rec}
\frac{ E_0 w_D(n,k)}{w_D(n,k)}=q^*\underset{c \in X(D)}{
\prod}\left(\frac{q^n}{q^a q^b} \right)^{\varepsilon(c)}
\frac{(1-q^{n-a'})(1-q^{b+1})}{(1-q^{n-a})(1-q^{b'+1})}
\end{equation}
\end{proposition}

\begin{remark} 
The denominators in the above equations actually vanish if $k \notin S_{D,n}$. 
To obtain recurrence relations that are valid for any $(n,k)$, we can 
simply move each denominator to the other side of the equation. The 
convention $\frac{1}{(q)_i}=0$ if $i<0$ will ensure that the equations 
still hold.
\end{remark}

\begin{proof}
Let us note first that the weights of local extrema are linear powers of $q$. 
When computing the ratio $\frac{E_0 w_D(n,k)}{w_D(n,k)}$ those weights 
will only contribute to a $q^*$ factor. Thus we can discard those weights 
while trying to prove Proposition~\ref{prop:E_0-rec}. We can also 
discard any linear power $q$ from the weights of crossing for the same 
reason.

We also note that one can separate the weights $w(c)$ of crossings 
into a product of two factors $w_>(c)$ and $w_<(c)$, where
$$
w_>(c)=\frac{(q)_{n-a}}{(q)_{n-a'}}\frac{(q)_b}{(q)_{b'}(q)_k}
$$
and 
$$
w_<(c)=\begin{cases} q^{(n+na+nb'-a'b'-ab)/2} \ \textrm{if} \ \varepsilon(c)=+1,
\\ (-1)^k q^{(-n-na'-nb+a'b+ab')/2} \ \textrm{if} \ \varepsilon(c)=-1.
\end{cases}
$$
where $a,a',b,b'$ are the colors of arcs neighboring the crossing $c$, 
following the convention described in Figure \ref{fig:CJPweights}.

Recall that $E_0$ shifts the color of all arcs up by $1$. Up to $q^*$, 
the ratio $\frac{ E_0 w_D(n,k)}{w_D(n,k)}$ is a product of factors 
$\mu(c)=\frac{ E_0 w_>(c)}{w_>(c)}$ and $\nu(c)
=\frac{ E_0 w_<(c)}{w_<(c)}$ for every crossing. We compute that:
$$
\mu(c)=\frac{(q)_{n-a-1}(q)_{n-a'}}{(q)_{n-a}(q)_{n-a'-1}}
\frac{(q)_{b+1}(q)_{b'}}{(q)_b (q)_{b'+1}}=
\frac{(1-q^{n-a'})(1-q^{b+1})}{(1-q^{n-a})(1-q^{b'+1})},
$$
and
$$
\nu(c)=q^*\frac{q^{(n(a+1)+n(b'+1)-(a'+1)(b'+1)-(a+1)(b+1))/2}}{
q^{(na+nb'-a'b'-ab)/2}}=
q^*\frac{q^n}{q^{(a+a'+b+b')/2}}=q^*\frac{q^n}{q^{a+b}},
$$
if $c$ is positive and
$$
\nu(c)=
q^*\frac{q^{(-n(a'+1)-n(b+1)+(a'+1)(b+1)+(a+1)(b'+1))/2}}{q^{(-na'-nb+a'b+ab')/2}}
=q^*\frac{q^{(a+a'+b+b')/2}}{q^n}=q^*\frac{q^{a+b}}{q^n},
$$
if $c$ is negative. This gives Equation~\eqref{eq:E_0-rec}.
\end{proof}

Let us now turn to the effect of operator $E$.

\begin{proposition}
\label{prop:CJPrec}
The summand $w_D(n,k)$ of the colored Jones state sum satisfies:
\begin{equation}
\label{eq:E-rec}
\frac{E w_D(n,k)}{w_D(n,k)}
=q^*\underset{X(D)}{
\prod}q^{\varepsilon(c) \left( \frac{a+b}{2}\right)-\frac{k}{2}}
\left(\frac{1-q^{n+1-a}}{1-q^{n+1-a'}}\right).
\end{equation}
\end{proposition}

\begin{proof}
Again, we can safely ignore the contribution of weights of local extrema 
and any linear power of $q$ in the weights of crossings as they just 
contribute to a $q^*$ factor.
First, note that the effect of $E$ is to shift $n$ up by $1$ 
and leave the colors of all arcs invariant. Then, as in the previous 
Proposition, any crossing $c$ contributes to the ratio by the product 
of two factors $\mu(c)$ and $\nu(c)$, where
$$
\mu(c)=\frac{E w_>(c)}{w_>(c)}=
\frac{(q)_{n+1-a}(q)_{n-a'}}{(q)_{n+1-a'}(q)_{n-a}}=
\frac{(1-q^{n+1-a})}{(1-q^{n+1-a'})},
$$
and 
$$
\nu (c)=\frac{E w_<(c)}{w_<(c)}=
q^*\frac{q^{\frac{(n+1)a+(n+1)b'-a'b'-ab}{2}}}{
q^{\frac{na+nb'-a'b'-ab}{2}}}=q^* q^{(a+b')/2}=q^* q^{(a+b)/2-k/2},
$$
as $b'=b-k$, if $c$ is a positive crossing. For $c$ a negative 
crossing, we have:
$$
\nu(c)=q^*\frac{(-1)^k q^{\frac{(-(n+1)a'-(n+1)b+a'b+ab')}{2}}}{
(-1)^k q^{\frac{(-na'-nb+a'b+ab')}{2}}}=q^* q^{(-a'-b)/2}=q^* q^{-(a+b)/2-k/2}.
$$
as $a'=a+k$. Combining the factors $\mu(c)$ and $\nu(c)$ we get 
Equation~\eqref{eq:E-rec}.
\end{proof}

\begin{proposition}
\label{prop:E_c-rec}
Fix a labeled diagram $D$ as in Section~\ref{sub.labelD}. 
Let $c$ be a crossing of $D$ with labels $j<j'$.
Then the summand $w_D(n,k)$ of the colored Jones polynomial satisfies:
\begin{equation}
\label{eq:E_c-rec-overpass}
\frac{E_c w_D(n,k)}{w_D(n,k)}=q^* F_c\underset{l \in O(D)\cap (j,j')}{\prod} 
\left(\frac{q^n}{q^b q^{b'}}\right)^{\frac{\varepsilon(l)}{2}}
\frac{1-q^{n-a'}}{1-q^{n-a}} \ 
\underset{l \in U(D) \cap (j,j')}{\prod} 
\left(\frac{q^n}{q^a q^{a'}}\right)^{\frac{\varepsilon(l)}{2}}
\frac{1-q^{b+1}}{1-q^{b'+1}},
\end{equation}
if $j$ is an overpass and
\begin{equation}
\label{eq:E_c-rec-underpass}
\frac{E_c w_D(n,k)}{w_D(n,k)}=q^* F_c \underset{l \in O(D)\cap (j,j')}{\prod} 
\left(\frac{q^n}{q^b q^{b'}}\right)^{-\frac{\varepsilon(l)}{2}}
\frac{1-q^{n+1-a}}{1-q^{n+1-a'}} 
\underset{l \in U(D) \cap (j,j')}{\prod} 
\left(\frac{q^n}{q^a q^{a'}}\right)^{-\frac{\varepsilon(l)}{2}}
\frac{1-q^{b'}}{1-q^b},
\end{equation}
if $j$ is an underpass. In the above, we set
$$
F_c=
\begin{cases}
q^{-\frac{a_c+b_c'}{2}}\left(\frac{(1-q^{b_c+1})(1-q^{n-a_c'})}{1-q^{k_c+1}}
\right)
& \text{if $j$ is an overpass and $\varepsilon (c)=+1$},
\\
-q^{\frac{a_c'+b_c}{2}-n}\left(\frac{(1-q^{b_c+1})(1-q^{n-a_c'})}{1-q^{k_c+1}}
\right)
& \text{if $j$ is an overpass and $\varepsilon (c)=-1$},
\\
q^{\frac{a_c'+b_c}{2}-n}\left(\frac{(1-q^{b_c'})(1-q^{n-a_c+1})}{1-q^{k_c+1}}
\right)
& \text{if $j$ is an underpass and $\varepsilon (c)=+1$}, 
\\
-q^{-\frac{a_c+b_c'}{2}}\left(\frac{(1-q^{b_c'})(1-q^{n-a_c+1})}{1-q^{k_c+1}}
\right)
& \text{if $j$ is an underpass and $\varepsilon (c)=-1$} \,.
\end{cases}
$$
\end{proposition}

\begin{proof}
Let $c$ be a crossing with labels $j<j'$. The effect of $E_c$ is to shift 
$k_c$ up by $1$. Note that the colors of arcs 
$[k,k+1]\subset [1,j] \cup [j',1]$ do not depend on $k_c$, while the 
colors of arcs $[k,k+1] \subset [j,j']$ are of the form 
$c_0+\varepsilon k_c$, where $c_0$ does not depend on $k_c$ and 
$\varepsilon =1$ if $j$ is an overpass, $\varepsilon=-1$ else. Thus the 
effect of $E_c$ is to shift the colors of arcs in $[j,j']$ up by $1$ 
(if $j$ is an overpass) or down by $1$ (if $j$ is an underpass).

As before we neglect the weights of local extrema and any linear power 
$q$ in the weights of crossings.
Let us write $a,a',b,b'$ for the colors of the arcs neighboring a 
crossing $c' \in (j,j')$ with labels $l<l'$, let $k=a'-a=b-b'$.

First we note that the weights 
$w_>(c')=\frac{(q)_{n-a}}{(q)_{n-a'}}\frac{(q)_b}{(q)_{b'}(q)_k}$ can 
separated into a factor $w_>(l)=\frac{(q)_{n-a}}{(q)_{n-a'}}$ associated 
to the overpass $l$ and a factor $w_>(l')=\frac{(q)_b}{(q)_{b'}(q)_{k}}$ 
associated to the underpass $l'$. The weights $w_<(c')$ are not separable 
in the same way; however the ratios $\nu(c')=\frac{E_c w_<(c')}{w_<(c')}$ 
are linear powers of $q$ and thus we can compute those factors up to $q^*$ 
as a product of two factors $\nu(l),\nu(l')$, where in $\nu(l)$ we apply 
the shift only to the colors $a,a'$ and in $\nu(l')$ we apply the shift 
only to the colors $b,b'$.

Now we compute the factors $\mu(l)=\frac{E_c w_>(l)}{w_>(l)}$ and 
$\nu(l)$ associated to over- or underpasses.

Note that if $l \notin \llbracket j,j' \rrbracket$, then no arc of 
the over- or underpass $l$ has its color changed under the shift $E_c$. 
Thus $\mu(l),\nu(l)=1$ in this case.

Consider $l \in (j,j')$ that corresponds to a positive crossing. 
Assume first that $l$ is an overpass. If $j$ is an overpass, the operator 
$E_c$ shifts the colors $a,a'$ up by $1$, and we have
$$
\mu(l)=\frac{(q)_{n-(a+1)}(q)_{n-a'}}{(q)_{n-(a'+1)}(q)_{n-a}}=
\frac{1-q^{n-a'}}{1-q^{n-a}}, \ \textrm{and} \ \nu(l)=
q^*\frac{q^{\frac{n(a+1)-(a'+1)b'-(a+1)b}{2}}}{q^{\frac{na-a'b-ab}{2}}}=
q^* q^{\frac{n-b-b'}{2}}.
$$
If $j$ was an underpass instead, colors $a,a'$ are shifted down by $1$ 
under $E_c$, so that
$$
\mu(l)=\frac{(q)_{n-(a-1)}(q)_{n-a'}}{(q)_{n-(a'-1)}(q)_{n-a}}=
\frac{1-q^{n+1-a}}{1-q^{n+1-a'}}, \ \textrm{and} \ \nu(l)=
q^*\frac{q^{\frac{n(a-1)-(a'-1)b'-(a-1)b}{2}}}{q^{\frac{na-a'b-ab}{2}}}=
q^* q^{-\frac{n-b-b'}{2}}.
$$
Now if $l \in (j,j')$ is an underpass and $j$ is an overpass, the colors 
$b,b'$ are shifted up by $1$ under $E_c$ and we get:
$$
\mu(l)=\frac{(q)_{b+1}(q)_{b'}}{(q)_{b'+1}(q)_b}=
\frac{1-q^{b+1}}{1-q^{b'+1}}, \ \textrm{and} \ \nu(l)=
q^*\frac{q^{\frac{n(b'+1)-a'(b'+1)-a(b+1)}{2}}}{q^{\frac{(nb'-a'b'-ab}{2}}}=
q^* q^{\frac{(n-a'-a)}{2}}.
$$
Finally if $j$ is an underpass instead, colors $b,b'$ are shifted down by 
$1$ and:
$$
\mu(l)=\frac{(q)_{b-1}(q)_{b'}}{(q)_{b'-1}(q)_b}=\frac{1-q^{b'}}{1-q^b}, 
\ \textrm{and} \ \nu(l)=q^*\frac{q^{\frac{n(b'-1)-a'(b'-1)-a(b-1)}{2}}}{
q^{\frac{nb'-a'b'-ab}{2}}}=q^* q^{\frac{-(n-a'-a)}{2}}.
$$
We see that those factors match with the ones in 
Equation \eqref{eq:E_c-rec-overpass} and \eqref{eq:E_c-rec-underpass} 
considering $\varepsilon(l)=+1$. If $l$ corresponds to a negative crossing, 
only the $\nu(l)$ factor is changed. The computation of the $\nu(l)$ 
factors is similar and left to the reader.

There is now just one factor to be considered: the factor 
$F_c=\frac{E_c w(c)}{w(c)}$ coming from crossing $c$. Assume that $j$ 
is a positive overpass, then $E_c$ shifts the colors $a_c'$ and $b_c$ 
up by one and leaves colors $a_c,b_c'$ invariant. Also here $E_c$ shifts 
$k_c$ up by one. We get
$$
\mu(c)=\frac{E_c w_>(c)}{w_>(c)}=\frac{(q)_{n-a_c'}}{
(q)_{n-(a_c'+1)}}\frac{(q)_{b_c+1}(q)_k}{(q)_{b_c}(q)_{k+1}}=
\frac{(1-q^{n-a_c'})(1-q^{b_c+1})}{(1-q^{k_c+1})}
$$
and 
$$
\nu(c)=\frac{E_c w_<(c)}{w_<(c)}=q^*\frac{q^{\frac{-(a_c'+1)b_c-a_c(b_c+1)}{2}}}{
q^{\frac{-a_c'b_c-a_cb_c}{2}}}=q^* q^{-\frac{a_c+b_c'}{2}}.
$$
Thus $F_c=\mu(c)\nu(c)$ matches with the formula of 
Proposition~\ref{prop:CJPrec}. The other possibilities for $j$ (negative 
overpass, positive underpass, negative underpass) yield similar 
computations and are left to the reader.
\end{proof}

Recall that the annihilator ideal $\Ann(w_D)$ is a left ideal of the
ring $\BQ[q,Q,Q_\vc] \langle E, E_\vc \rangle$ where 
$Q_\vc=(Q_0,\dots,Q_{c(D)})$ and $E_\vc=(E_0,\dots,E_{c(D)})$. 
Let $\Ann_\rat(w_D)$ denote the corresponding ideal of the ring
$\BQ(q,Q,Q_\vc)\langle E, E_\vc \rangle$. Let $R$, $R_c$ 
(for $c=1,\dots, c(D)$) and $R_0$ denote the expressions on the right
hand side of Equations~\eqref{eq:E-rec}, \eqref{eq:E_0-rec} and
\eqref{eq:E_c-rec-overpass}, \eqref{eq:E_c-rec-underpass} 
respectively. 

\begin{proposition}
\label{prop.ann}
The ideal $\Ann_\rat(w_D)$ is generated by the set
\be
\label{eq.genann}
\{E_c-R_c(q,Q,Q_c), \,\, c=0, \cdots, c(D), \,\, E-R_c(q,Q,Q_c) \} \,.
\ee
\end{proposition}

Below, we will need to specialize our operators to $q=1$. To make this
possible, we introduce the subring $Q_\loc(q,Q,Q_c)$ of the field 
$\BQ(q,Q,Q_c)$ that consists of all rational functions that are regular 
(i.e., well-defined) at $q=1$. 

Let $\Ann_{\rat,\loc}(w_D) = \Ann_\rat(w_D) \cap 
\BQ_\loc(q,Q,Q_c) \langle E, E_c \rangle$
denote the left ideal of the ring $\BQ_\loc(q,Q,Q_c) \langle E, E_c \rangle$.

\begin{proposition}
\label{prop.ann2}
The ideal $\Ann_{\rat,\loc}(w_D)$ is generated by the set~\eqref{eq.genann}.
\end{proposition}

\begin{proof}
First, let us note that $\BQ_\loc(q,Q,Q_\vc)\langle E,E_\vc \rangle$ is a 
subring of $\BQ(q,Q,Q_\vc)\langle E,E_\vc \rangle.$

Indeed, if $P(q,Q,Q_\vc)$ is in $\BQ_\loc(q,Q,Q_\vc)$ then 
$E P(q,Q,Q_\vc)=P(q,q^{-1}Q,Q_\vc) E$ is also in 
$\BQ_\loc(q,Q,Q_\vc)\langle E,E_\vc \rangle,$ as the denominator of 
$P(q,q^{-1}Q,Q_\vc)$ evaluated at $q=1$ is the same as that of $R(q,Q,Q_\vc)$. 
The same can said for multiplication by one of the $E_\vc$'s.

Secondly, it is easy to see that the elements $R(q,Q,Q_\vc)$ and 
$R_c(q,Q,Q_\vc)$ (for $c=1,\dots,c(D)$) are in $\BQ_\loc(q,Q,Q_\vc)$.
Let $I$ be the left $\BQ_\loc(q,Q,Q_\vc)\langle E,E_\vc \rangle$ ideal 
generated by those elements.

Let us order monomials in $E$ and the $E_\vc$'s using a lexicographic order. 
We claim that $I$ contains a monic element in each non zero 
$(E,E_\vc)$-degree. Indeed, if $E-R(q,Q,Q_\vc)$ is one of the above described 
generators and $(\alpha,\beta_\vc)\in \BN^{c(D)+2}$, multiplying by 
$E^{\alpha}E_\vc^{\beta_\vc}$ on the left we get that $I$ contains an element 
of the form $E^{\alpha+1}E_\vc^{\beta_\vc} -\tilde{R}(q,Q,Q_\vc)E^{\alpha}E_\vc^{\beta_\vc}$ 
where $\tilde{R}(q,Q,Q_\vc) \in \BQ_\loc(q,Q,Q_\vc)$. Using also the generators 
$E_\vc -R_c(q,Q,Q_\vc)$ the claim follows.

Now let $P(q,Q,Q_\vc,E,E_\vc)$ be an arbitrary element $\Ann_{\rat,\loc}(w_D)$. 
We may write
$$
P(q,Q,Q_\vc,E,E_\vc)=\underset{(\alpha,\beta_\vc)\in \BN^{c(D)+2}}{
\sum}R_{\alpha,\beta_\vc}(q,Q,Q_\vc) E^{\alpha}E_\vc^{\beta_\vc} \,.
$$
As $I$ contains a monic element in each non-zero $(E,E_\vc)$ degree, one 
may subtrack elements of $I$ to $P(q,Q,Q_\vc,E,E_\vc)$ to drop its degree 
until we get that $P-S \in \BQ_\loc(q,Q,Q_\vc)$ for some element $S\in I$. 
But $P-S$ must also be in $\Ann_\loc(w_D)$, and as $w_D\neq 0$ it must be zero.
Thus we can conclude that $I=\Ann_{\rat,\loc}(w_D)$.
\end{proof}


\section{Matching the annihilator ideal and the gluing equations}
\label{sec.match}

\subsection{From the annihilator of the state summand to the 
gluing equations variety}

In the previous sections we studied the gluing equations variety $\calG_D$
of a knot diagram $D$ and the state summand $w_D(n,k)$ of the colored Jones
polynomial of $K$. In this section we compare the annihilator ideal
of the summand with the defining ideal of the gluing equations variety,
once we set $q=1$, and conclude that they exactly match.
Let us abbreviate the evaluation of a rational function
$f(q)$ at $q=1$ by  $\ev_q f(q)=f(1)$.

Consider the map $\psi$ defined by:

\be
\label{eq.psi}
\psi: \BQ[Q,Q_\vc][E] \to \BC[\calG_D], \qquad (E, Q, Q_c) \mapsto 
(w_{\lambda}^{-1/2}, w_{\mu}, w_c) 
\ee
where $\BC[\calG_D]$ denotes the coordinate ring of the affine variety
$\calG_D$ and $w_{\lambda}^{-1/2}$ is the element of $\BC[\calG_D]$ described 
in Proposition~\ref{prop:sqroot}. 

The main result which connects the quantum invariant with the classical 
one can be summarized in the following.

\begin{theorem}
\label{thm.match}
{\rm (a)} We have:
\be
(\psi \circ \ev_q \circ \varphi )(\Ann_{\rat,\loc}(w_D)) = 0 \,.
\ee
{\rm (b)} If 
$P(q,Q,E) \in \varphi( \Ann(F) \cap \BQ[q,Q]\langle E, E_k \rangle )$
as in~\eqref{eq.cert}, then $P(q,Q,E)$ annihilates the colored Jones 
polynomial and $P(1,w_\mu,w_\lambda^{-1/2})=0$.
\end{theorem}

\begin{proof}
Recall the generators of the annihilator ideal 
$\Ann_{\rat,\loc}(w_D)$ given by Equation~\eqref{eq.genann}, as well as
the functions $L_k-1$ for $k=0,\dots,c(D)$ of the coordinate ring of 
$\calG_D$ defined in Section~\ref{sub.loop}. We will match the two. 

For an arc of the diagram with color $a$, let $Q_a$ be the multiplication 
by $q^a$. We claim that $\varphi(Q_a)=z_a$, the corresponding arc parameter. 
Indeed, for the arc $[1,2]$ we have $\varphi(Q_0)=w_0$ is the arc parameter 
of the arc $[1,2]$, and going from arc $[k-1,k]$ to $[k,k+1]$ we shift the 
multiplication operator by $Q_c^{\pm 1}$ and the arc parameter by $w_c^{\pm 1}$, 
depending on whether $k$ is an over- or underpass.

By Equations \eqref{eq:E-rec} and \eqref{eq:E_0-rec}:
\begin{align*}
R(1,Q,Q_\vc) &=\underset{X(D)}{\prod}\left( 
\frac{1-\frac{Q}{Q_a}}{1-\frac{Q}{Q_{a'}}}\right)
\left(Q_a Q_b \right)^{\frac{\varepsilon(c)}{2}}Q^{-\frac{1}{2}} 
\\
R_0(1,Q,Q_\vc) &
=\underset{c \in X(D)}{\prod}\left( \frac{Q}{Q_a Q_b}
\right)^{\varepsilon(c)} \frac{(1-\frac{Q}{Q_{a'}})(1-Q_b)}{
(1-\frac{Q}{Q_a})(1-Q_{b'})} \,.
\end{align*}

If $c$ is a crossing with labels $j<j'$, and $j$ is an overpass, we have 
by Equation~\eqref{eq:E_c-rec-overpass}:
\begin{align*}
R_c(1,Q,Q_\vc) &=\ev_q(F_c)\underset{k \in (j,j')  
\cap O(D)}{\prod}\left(\frac{Q}{Q_b Q_{b'}}\right)^{\frac{\varepsilon(k)}{2}}
\left( \frac{1-\frac{Q}{Q_{a'}}}{1-\frac{Q}{Q_a}}\right) 
\\ & \quad \times \underset{k \in (j,j')\cap U(D)}{\prod}
\left(\frac{Q}{Q_a Q_{a'}}\right)^{\frac{\varepsilon(k)}{2}}\left(\frac{1-Q_b}{
1-Q_{b'}}\right)
\end{align*}
where 
$$
\ev_q(F_c)=\frac{1}{Q_{b_c'}}
\frac{(1-\frac{Q}{Q_{a_c'}})(1-Q_{b_c})}{(1-Q_c)}
$$ 
if $c$ is a positive crossing for example.
Similarly by Equation~\eqref{eq:E_c-rec-underpass}, if $j$ is an underpass, 
then 
\begin{align*}
R_c(1,Q,Q_\vc)) &=\ev_q(F_c) \underset{l \in O(D)\cap (j,j')}{\prod} 
\left(\frac{Q}{Q_b Q_{b'}}\right)^{-\frac{\varepsilon(k)}{2}}
\left( \frac{1-\frac{Q}{Q_{a}}}{1-\frac{Q}{Q_{a'}}}\right) 
\\ & \quad \times \underset{l \in U(D) \cap (j,j')}{\prod} 
\left(\frac{Q}{Q_a Q_{a'}}\right)^{-\frac{\varepsilon(k)}{2}}
\left(\frac{1-Q_{b'}}{1-Q_b}\right).
\end{align*}

Comparing $(\ev_q \circ \varphi)(E-R(q,Q,Q_\vc))$ with 
Equation~\eqref{eq:sqroot}, we get that 
$$
(\ev_q \circ \varphi)(E-R(q,Q,Q_\vc))=s-s=0 \,.
$$

Comparing $(\ev_q \circ \varphi)(E_0-R_0(q,Q,Q_\vc))$ with 
Equation~\eqref{eq:loopeq3}, we get that 
$$
(\ev_q \circ \varphi)(E_0-R_0(q,Q,Q_\vc))=1-L_0 \,.
$$

Finally, if $c$ is a crossing with labels $j<j'$, comparing  
$(\ev_q \circ \varphi)(E_c-R_c(q,Q,Q_\vc))$ with 
Equation~\eqref{eq:loopequation2}, 
we get that 
$$
(\ev_q \circ \varphi)(E_c-R_c(q,Q,Q_\vc))=1-L_c\,
$$ 
if $j$ is an overpass, while if $j$ is an underpass we get that 
$$
(\ev_q \circ \varphi)(E_c-R_c(q,Q,Q_\vc))=1-L_c^{-1}=L_c^{-1}(L_c-1) \,.
$$
Thus the image of the generators of the ideal $\Ann_{\rat,\loc}(w_D)$ 
by $\ev_q \circ \varphi$ are generators of the ideal $I_D$. This proves
part (a) of Theorem~\ref{thm.match}. Part (b) follows from part (a)
and Equation~\eqref{eq.cert}.
\end{proof}  

\subsection{Proof of Theorem~\ref{thm.1}}
\label{sub.proof.thm1}

\begin{proof}
Fix a labeled, oriented planar projection $D$ of $K$. Then, 
the certificate recursion $\hat{A}^c_D(q,Q,E)$ annihilates the colored Jones 
polynomial, as this is true for all $q$-holonomic sums~\eqref{eq.fF}. This
concludes part (a).
 
For part (b), Theorem~\ref{thm.match} implies that 
$\hat{A}^c_D(1,w_\mu,w_\lambda^{-1/2}) =0 \in \BC[\calG_D]$. In other words,
the function $\hat{A}^c_D(1,w_\mu,w_\lambda^{-1/2})$ in the coordinate ring of
$\calG_D$ is identically zero. Since this is true for every labeled, oriented
diagram $D$ of a knot $K$, this concludes part (b) of Theorem~\ref{thm.1}.
\end{proof}

\subsection*{Acknowledgements} 
S.G. was supported in part by DMS-18-11114. S.G. wishes to thank Dylan
Thurston for enlightening conversations in several occasions regarding 
the octahedral decomposition of knot complements and the state sum 
formulas for the colored Jones polynomial and Christoph Koutchan for
enlightening conversations on $q$-holonomic functions. 
The paper was conceived and completed while both authors were visiting
the Max-Planck Institute for Mathematics in Bonn. The authors wish to
thank the institute for its hospitality. 


\bibliographystyle{hamsalpha}
\bibliography{biblio}
\end{document}